\documentclass[11pt]{amsart}
\usepackage{amsmath,amssymb,amscd,graphicx,psfrag,epsf,amsthm}
\usepackage[all]{xy}

\newtheorem{thm}{Theorem}[section]
\newtheorem*{thm*}{Theorem}
\newtheorem{question}[thm]{Question}
\newtheorem{lemma}[thm]{Lemma}
\newtheorem{cor}[thm]{Corollary}
\newtheorem{prop}[thm]{Proposition}

\theoremstyle{definition}

\newtheorem{example}[thm]{Example}

\newtheorem{defn-thm}[thm]{Definition-Theorem}  
\newtheorem{claim}[thm]{Claim} 
\newtheorem{rmk}[thm]{Remark} 
\newtheorem{questions}[thm]{Questions}

\newcommand{\PP}{\text {\bf P}}
\newcommand{\QQ}{\text {\bf Q}}
\newcommand{\RR}{\text {\bf R}}
\newcommand{\FF}{\text {\bf F}}
\newcommand{\ZZ}{\text {\bf Z}}

\newcommand{\ch}{\text {ch}}
\newcommand{\ev}{\text {ev}}

\newcommand{\cE}{{\mathcal {E}}}
\newcommand{\cO}{{\mathcal {O}}}
\newcommand{\cQ}{{\mathcal {Q}}}
\newcommand{\cS}{{\mathcal {S}}}

\newcommand{\tT}{{\tilde{T}}}
\newcommand{\tX}{{\tilde{X}}}

\newcommand{\X}{{\mathcal {X}}}

\def\Cl{\text{Cl}}
\def\NE{\overline{\text{NE}}}
\def\eff{\overline{\text{NE}}}
\def\Pic{\text{Pic}}
\def\Proj{\text{Proj}}
\def\Sym{\text{Sym}}
\newcommand{\rat}[0]{\operatorname{RatCurves}^n}
\newcommand{\onto}[0]{\twoheadrightarrow}

\def\dim{\text{dim}}
\def\codim{\text{codim}}

\newcommand{\ra}[0]{\rightarrow}
\newcommand{\hra}[0]{\hookrightarrow}

\numberwithin{equation}{section}

\title{Classification of 2-Fano manifolds with high index}

\date{}

\author{Carolina Araujo}
\author{Ana-Maria Castravet} 

\address{Carolina Araujo: \sf IMPA, Estrada Dona Castorina 110, Rio de
  Janeiro, RJ 22460-320, Brazil} 
\email{caraujo@impa.br}

\address{\vskip -.5cm Ana-Maria Castravet: \sf Department of Mathematics, 
The Ohio State University 100 Math Tower, 231 West 18th Avenue, Columbus, OH 43210-1174, USA} 
\email{noni@alum.mit.edu}

\begin{document}

\maketitle

\def\thefootnote{}


\begin{center}

{\it Dedicated to Joe Harris. }
\end{center} 

\begin{abstract}
In this paper we classify 
$n$-dimensional Fano manifolds with index $\geq n-2$ and
positive  second  Chern character.
\end{abstract}

\tableofcontents

%%%%%%%%%%%%%%%%%%%%%%%%%%%%%%%%%%%%%%%%%%%%%%%%%%%%%%%%%
%                                                                           %
%               SECTION 1 - INTRODUCTION                           %
%                                                                           %
%%%%%%%%%%%%%%%%%%%%%%%%%%%%%%%%%%%%%%%%%%%%%%%%%%%%%%%%%

\section{Introduction}\label{intro}

A Fano manifold is a smooth complex projective variety $X$ having ample anticanonical class, 
$-K_X>0$.  This simple  condition  has far reaching geometric implications. 
For instance,  any
Fano manifold $X$ is \emph{rationally connected}, i.e., there are rational curves connecting 
any two points of $X$ (\cite{campana} and \cite{kmm3}).

The Fano condition $-K_X>0$  also plays a distinguished role in arithmetic geometry. 
In the landmark paper \cite{GHS}, Graber, Harris and Starr showed that 
proper families of rationally connected varieties over curves 
always admit sections.
This generalizes Tsen's theorem in the case of function fields of curves.

\begin{thm*}[Tsen's Teorem] 
Let $K$ be a field of transcendence degree $r$ over an algebraically closed field $k$.
Let $\X\subset \PP_K^{n}$ be a hypersurface of degree $d$. If $d^r\leq n$, then $\X$ has a $K$-point.
\end{thm*}

For hypersurfaces of degree $d$ in $\PP^n$, being Fano or rationally connected is equivalent to 
the numerical condition $d\leq n$.
So, for $r=1$, \cite{GHS} replaces the condition of $\X$ being a hypersurface of degree $d\leq n$
with the condition of $\X$ being rationally connected.
It turns out that rationally connected varieties form the largest class of varieties for which such 
statement holds true when $r=1$ (see \cite{GHS_converse} for the precise statement).

Since then, there has been quite some effort towards finding suitable geometric conditions on 
$\X$ that generalize  
Tsen's theorem for function fields of higher dimensional varieties. 
In \cite{dJ-S:sections_over_surfaces}, de Jong and Starr considered a possible notion of 
\emph{rationally simply connectedness}. They established a version of Tsen's theorem 
for function fields of surfaces, 
replacing the condition of $\X$ being a hypersurface of degree $d$, $d^2\leq n$, with
the condition of $\X$ being rationally simply connected (see 
\cite[Corollary 1.1]{dJ-S:sections_over_surfaces} for a precise statement). 
Several attempts have been made to define the appropriate notion of rationally simply connectedness.
Roughly speaking, one would like to ask that a suitable irreducible component of the space of rational
curves through two general points of $X$ is itself rationally connected. 
However, in order to make the definition applicable, one is led to introduce some technical 
hypothesis, which
makes this  condition difficult to verify in practice. 
It is then desirable to have natural geometric conditions that imply rationally simply connectedness. 
In this context, \emph{2-Fano manifolds} were introduced by de Jong and Starr  in 
\cite{dJ-S:2fanos_1} and \cite{dJ-S:2fanos_2}.
In order to define these, we  introduce some notation. 
Given a smooth projective variety $X$ and a positive integer $k$, 
we denote by $N_k(X)$ the $\RR$-vector space of $k$-cycles on $X$
modulo numerical equivalence, and by $\NE_k(X)$ the closed convex cone in $N_k(X)$
generated by classes of effective $k$-cycles.
Recall that the second Chern character of $X$ is 
$$
\ch_2(X)=\frac{c_1(X)^2}{2}-c_2(X),
$$
where $c_i(X)=c_i(T_X)$. 
We say that a manifold $X$ is \emph{2-Fano} (respectively \emph{weakly 2-Fano})
if it is Fano and $\ch_2(X) \cdot \alpha>0$ 
(respectively $\ch_2(X) \cdot \alpha\geq 0$)
for every $\alpha\in \NE_2(X)\setminus \{0\}$.

\begin{questions}\label{conj}
Do 2-Fano manifolds satisfy some version of rationally simply connectedness?
Is this a good condition to impose on the general member of 
fibrations over surfaces in order to prove existence of rational sections
(modulo the vanishing of Brauer obstruction)?
\end{questions}

Motivated by these questions, in \cite{AC}, we investigated and classified certain spaces of rational curves on 2-Fano manifolds, and gave evidence for a positive answer to Questions~\ref{conj}.
In that work, we announced the following threefold classification.

\begin{thm}\label{main_thm}
The only 2-Fano threefolds are $\PP^3$ and the smooth quadric hypersurface $Q^3\subset \PP^4$.
\end{thm}

In this paper we write down a complete proof of Theorem~\ref{main_thm}.
In fact, Theorem~\ref{main_thm} will follow from a more general classification. 
Recall that the
index  $i_X$ of a Fano manifold $X$ is the largest integer dividing $-K_X$ in $\Pic(X)$. 
Our main result is the following.

\begin{thm}\label{main_thm_dim=n}
Let $X$ be a 2-Fano manifold of dimension $n\geq 3$ and index $i_X\geq n-2$.
Then $X$ is isomorphic to one of the following.
\begin{itemize}
\item $\PP^n$.

\item Complete intersections in projective spaces:
\begin{itemize}
\item[-]  Quadric hypersurfaces $Q^n\subset \PP^{n+1}$ with $n>2$;
\item[-]  Complete intersections of quadrics $X_{2\cdot2}\subset\PP^{n+2}$ with  $n>5$;
\item[-]  Cubic hypersurfaces $X_3\subset\PP^{n+1}$ with $n>7$;
\item[-]  Quartic hypersurfaces in $\PP^{n+1}$ with $n>15$;
\item[-]  Complete intersections $X_{2\cdot3}\subset\PP^{n+2}$ with $n>11$; 
\item[-]  Complete intersections $X_{2\cdot2\cdot2}\subset\PP^{n+3}$ with $n>9$.  
\end{itemize}

\item Complete interesections in weighted projective spaces:
\begin{itemize}
\item[-]  Degree $4$ hypersurfaces in $\PP(2,1,\ldots,1)$ with $n>11$; 
\item[-]  Degree $6$ hypersurfaces in $\PP(3,2,1,\ldots,1)$ with $n>23$; 
\item[-]  Degree $6$ hypersurfaces in $\PP(3,1,\ldots,1)$ with $n>26$; 
\item[-] Complete intersections of two quadrics in $\PP(2,1,\ldots,1)$ with $n>14$. 
\end{itemize}

\item $G(2,5)$.
\item $OG_+(5,10)$ and its linear sections of codimension $c<4$. 
\item $SG(3,6)$.
\item $G_2/P_2$.
\end{itemize}
\end{thm}

Here $OG_+(5,10)$ denotes a connected component of the $10$-dimensional orthogonal Grassmannian $OG(5,10)$ in the half-spinor embedding (see Section~\ref{OG(k,n)}), $SG(3,6)$ is a $6$-dimensional symplectic Grassmannian (see Section~\ref{SG}), and $G_2/P_2$ is a $5$-dimensional homogeneous variety for a group of type $G_2$ (see Section~\ref{ci's in G_2/P_2}).

In order to prove Theorem~\ref{main_thm_dim=n},
we will go through the classification of Fano 
manifolds of dimension $n\geq 3$ and index $i_X\geq n-2$,
and check positivity of the second  Chern character for each of them.
In the course of the proof, 
we also determine which of these manifolds are weakly 2-Fano.

The paper is organized as follows. 
In Section~\ref{Section:classification_of_Fanos} we revise 
the classification of Fano manifolds of high index.
In Section~\ref{Section:first_examples},
we check the 2-Fano condition for 
the simplest ones: 
(weighted) projective spaces and complete intersections on them, and Grassmannians. 
Most of the others can be described as double covers, blow-ups or projective bundles
over simple ones.
So in Section~\ref{Section:computations}
we compute Chern characters for these constructions.
In Section~\ref{Section:rational_curves}, we revise 
some results from \cite{AC}, which describe
certain families of rational curves on 2-Fano manifolds.
These results are then used in Section~\ref{Section:ci_in homogeneous}
to check  the 2-Fano condition for certain Fano manifolds 
described as complete intersections on homogeneous
spaces.
After all these computations, we are ready to prove Theorem~\ref{main_thm_dim=n}.
In Section~\ref{Section:2-Fano_classification_rho=1}, 
we address $n$-dimensional Fano manifolds
with index $i_X\geq n-2$, except Fano threefolds 
and fourfolds with Picard number $\geq 2$.
These are treated in Sections~\ref{Section:classification_rho>1}
and \ref{Section:4folds} respectively.
 
We remark that toric 2-Fano manifolds have been addressed in 
\cite{nobili_toric_2-fano_4folds}, \cite{sato_toric_2Fanos} and \cite{nobili_thesis}.
The only known examples are projective spaces. 

\medskip

\noindent {\bf Notation.} 
Given a vector bundle $\cE$ on a variety $X$,
we denote by $\PP_X(\cE)$, or simply $\PP(\cE)$, the projective bundle of one-dimensional quotients 
of the fibers of $\cE$, i.e., $\PP(\cE)=\Proj(\Sym \cE)$.

We denote by $G(k,n)$ the Grassmannian of $k$-dimensional subspaces of an 
$n$-dimensional vector space $V$, and we always assume that $2\leq k\leq\frac{n}{2}$.
We write 
$$0\ra\cS\ra\cO\otimes V\ra\cQ\ra 0$$
for the universal sequence on $G(k,n)$. 
For subvarieties $X$ of $G(k,n)$, we denote by the same symbols $\sigma_{a_1,\ldots, a_k}$ 
the restrictions to $X$ of the corresponding Schubert cycles. 

\

\noindent {\bf Aknowledgements.} 
The first named author was partially supported by CNPq and Faperj Research Fellowships.
The second named author was partially supported  by the NSF grants DMS-1001157 and  DMS-1160626.

%%%%%%%%%%%%%%%%%%%%%%%%%%%%%%%%%%%%%%%%%%%%%%%%%%%%%%%%%
%                                                       %
%                 SECTION 2                             %
%                                                       %
%%%%%%%%%%%%%%%%%%%%%%%%%%%%%%%%%%%%%%%%%%%%%%%%%%%%%%%%%

\section{Classification of Fano manifolds}\label{Section:classification_of_Fanos}

In this section we discuss the classification of Fano manifolds.
A modern survey on this subject  can be found in \cite{IP}. 

\medskip

\noindent {\bf Notation.}  
When $X$ is an $n$-dimensional Fano manifold with $\rho(X)=1$, we denote by 
$L$ the ample generator of $\Pic(X)$, and define the degree of $X$ as $d_X:=c_1(L)^n$.

\medskip

For a fixed positive integer $n$, Fano $n$-folds form a bounded family (see \cite{kmm2}).
For $n\leq 3$, Fano $n$-folds are completely classified. 
The classification of Fano surfaces, also known as \emph{Del Pezzo surfaces}, is a classical 
result. They are $\PP^2$, $\PP^1\times\PP^1$, and the blow-up $S_{9-n}$ of $\PP^2$ at $n$ points 
in general position, $1\leq n\leq 8$.
It is easy to check that among those only $\PP^2$ is 2-Fano, and among the others only
$S_8=\FF_1$ and $\PP^1\times\PP^1$ are weakly $2$-Fano (see \ref{del Pezzo}).

The classification of Fano threefolds of Picard number $\rho=1$ was established
by Iskovskikh in \cite{Isk1} and \cite{Isk2}. There are 17 deformation types of these.
%(see Theorem~\ref{Thm:Fano_3folds} below).
The classification of  Fano threefolds of Picard number $\rho\geq2$ was established
by Mori and Mukai in \cite{Mori-Mukai} and \cite{Mori-Mukai erratum}. There are $88$ 
deformation types of those. We will revise this list in Section~\ref{>1}.

In higher dimensions, there is no complete classification. 
On the other hand, one can get results in this direction if one fixes some invariants 
of the Fano manifold. For instance, we have the following result by Wi\'sniewski.

\begin{thm}[\cite{wis}]\label{Thm:Wis}
Let $X$ be an
$n$-dimensional Fano manifold
with index $i_X \geq \frac{n+1}{2}$.
Then $X$ satisfies one of the following conditions:
\begin{itemize}
\item $\rho(X)=1$;
\item $X \cong \PP^{\frac{n}{2}} \times \PP^{\frac{n}{2}}$ ($n$ even);
\item $X \cong \PP^{\frac{n-1}{2}} \times Q^{\frac{n+1}{2}}$ ($n$ odd);
\item $X \cong \PP(T_{\PP^{\frac{n+1}{2}}})$ ($n$ odd); or
\item $X \cong \PP_{\PP^{\frac{n+1}{2}}}(\mathcal{O}(1) \oplus \mathcal{O}^{\frac{n-1}{2}})$ ($n$ odd).
\end{itemize}
\end{thm}

Fano manifolds of dimension $n$ and
index $i_X\geq n-2$ have been classified.
A classical result of Kobayachi-Ochiai's asserts that
$i_X\leq n+1$, and equality 
holds if and only if $X\simeq \PP^n$. Moreover, $i_X= n$ if and only if $X$ is a 
quadric hypersurface $Q^n\subset \PP^{n+1}$ (\cite{kobayashi_ochiai}). 
Fano manifolds with index $i_X= n-1$ 
are called \emph{Del Pezzo} manifolds. 
They were classified by Fujita in \cite{fujita82a} and  \cite{fujita82b}:

\begin{thm}\label{del Pezzo varieties}
Let $X$ be an $n$-dimensional Fano manifold with index $i_X=n-1$, $n\geq 3$.
\begin{enumerate}
	\item Suppose that $\rho(X)=1$. Then $1\leq d_X\leq 5$. Moreover, 
	for each  $1\leq d\leq 4$ and $n\geq 3$, and for $d=5$ and $3\leq n\leq 6$,
	there exists a unique deformation class 
	of $n$-dimensional Fano manifolds $Y_d$ with $\rho(X)=1$, $i_X=n-1$ and $d_X=d$. 
	They have the following description:
\begin{itemize}
\item[(i) ] $Y_5$ is a linear section of the Grassmannian $G(2,5)\subset\PP^9$ (embedded via the Pl\"ucker embedding).

\item[(ii) ] $Y_4=Q\cap Q'\subset\PP^{n+2}$ is an intersection of two quadrics in $\PP^{n+2}$.

\item[(iii) ] $Y_3\subset\PP^{n+1}$ is a  cubic hypersurface.

\item[(iv) ] $Y_2\ra\PP^n$ is a double cover branched along a quartic $B\subset\PP^n$ (alternatively, $Y_2$ is a hypersurface of degree $4$ in the weighted projective space $\PP(2,1,\ldots,1)$).

\item[(v) ] $Y_1$ is a hypersurface of degree $6$ in the weighted projective space $\PP(3,2,1,\ldots,1)$.
\end{itemize}

	\item Suppose that $\rho(X)>1$. Then $X$ is isomorphic to one of the following:
\begin{itemize}
\item $\PP^2\times\PP^2$ ($n=4$);
\item $\PP(T_{\PP^2})$ ($n=3$);
\item $\PP\big(\cO_{\PP^2}(1)\oplus \cO_{\PP^2}\big)$ ($n=3$); or
\item $\PP^1\times\PP^1\times\PP^1$ ($n=3$).
\end{itemize}
\end{enumerate} 
\end{thm}

An $n$-dimensional Fano manifold $X$ with index $i_X= n-2$ 
is called a \emph{Mukai} manifold. 
The classification of such manifolds was first announced in \cite{Mukai89}
(see also \cite{IP} and references therein). 
% \cite[Corollary 2.1.17 and Theorem 5.2.3]{IP}, \cite[Proposition 1 and Rmk. 2]{Mukai89}.
First we note  that, by Theorem~\ref{Thm:Wis}, 
if $n \geq 5$, then
$n$-dimensional Mukai manifolds have Picard number $\rho=1$,
except in the cases of $\PP^3\times \PP^3$,
$\PP^2 \times Q^3$, $\PP (T_{\PP^3})$ and $\PP_{\PP^3}(\cO(1)\oplus \cO)$.

So we start by considering $n$-dimensional Mukai manifolds $X$ with $\rho(X)=1$. 
In this case there is an integer $g=g_X$, 
called the \emph{genus} of $X$, such that $d_X=c_1(L)^n=2g-2$.
The linear system $|L|$ 
determines a morphism 
$$
\phi_{|L|}: X\ra\PP^{g+n-2}, 
$$ 
which is an embedding if $g\geq4$ (see \cite[Theorem 5.2.1]{IP}).

\begin{thm}\label{Mukai varieties}
Let $X$ be an $n$-dimensional Mukai manifold with $\rho(X)=1$.
Then $X$ has genus 
$g\leq12$ ($g\neq 11$) and we have the following descriptions.
\begin{enumerate}
\item If $g=12$, then $n=3$ and 
$X$ is the zero locus of a global section of the vector bundle
$\wedge^2\cS^*\oplus\wedge^2\cS^*\oplus\wedge^2\cS^*$ on the Grassmannian $G(3,7)$.

\medskip

\item If $6\leq g\leq 10$, then $X$ is a linear section of a variety  
$$\Sigma^{n(g)}_{2g-2}\subset\PP^{g+n(g)-2}$$
of dimension $n(g)$ and degree $2g-2$, which can be described as follows:

\begin{itemize}
\item[(g=6) ] $\Sigma^6_{10}\subset\PP^{10}$ is a quadric section of the cone over the 
Grassmannian $G(2,5)\subset\PP^{9}$ in the Pl\"ucker embedding; 

\item[(g=7) ] $\Sigma^{10}_{12}=OG_+(5,10)\subset\PP^{15}$ is a connected component of the orthogonal Grassmannian $OG(5,10)$ in the half-spinor embedding;

\item[(g=8) ] $\Sigma^8_{14}=G(2,6)\subset\PP^{14}$ is the  Grassmannian $G(2,6)$ in the Pl\"ucker embedding;

\item[(g=9) ] $\Sigma^6_{16}=SG(3,6)\subset\PP^{13}$ is the symplectic Grassmannian $SG(3,6)$ in the Pl\"ucker embedding;

\item[(g=10) ] $\Sigma^5_{18}=(G_2/P_2)\subset\PP^{13}$ is a $5$-dimensional homogeneous variety for a group of type $G_2$.
\end{itemize}

\medskip

\item If $g\leq5$, and the map 
$\phi_{|L|}$ is an embedding, then $X$ is a complete intersection as follows:

\medskip

\begin{itemize}
\item[(g=3) ] $X_4\subset\PP^{n+1}$ a quartic hypersurface;
\item[(g=4) ] $X_{2\cdot3}\subset\PP^{n+2}$ a complete intersection of a quadric and a cubic; 
\item[(g=5) ] $X_{2\cdot2\cdot2}\subset\PP^{n+3}$ a complete intersection of three quadrics. 
\end{itemize}

\medskip

\item If $g\leq3$, and the map 
$\phi_{|L|}$ is not an embedding, then:

\medskip

\begin{itemize}
\item[(g=2) ] $\phi_{|L|}: X\ra\PP^n$ is a double cover branched along a sextic
(alternatively, $X$ is a degree $6$ hypersurface in the weighted projective space $\PP(3,1,\ldots, 1)$);
\item[(g=3) ] $\phi_{|L|}: X\ra Q^n\subset\PP^{n+1}$ is a double cover branched along the intersection of $Q$ with a quartic hypersurface 
(alternatively, $X$ is a complete intersection of two quadric hypersurfaces in the weighted projective space $\PP(2,1,\ldots, 1)$). 
\end{itemize}
\end{enumerate}
\end{thm}

We will go through the classification of  Mukai manifolds with Picard number $\rho\geq2$  
and dimension $n\in \{3,4\}$
in Sections~\ref{>1} and \ref{Section:4folds}.

%%%%%%%%%%%%%%%%%%%%%%%%%%%%%%%%%%%%%%%%%%%%%%%%%%%%%%%%%
%                                                       %
%                 SECTION 3                             %
%                                                       %
%%%%%%%%%%%%%%%%%%%%%%%%%%%%%%%%%%%%%%%%%%%%%%%%%%%%%%%%%

\section{First Examples}\label{Section:first_examples}

In this section we compute the first Chern characters 
for the simplest examples of Fano manifolds with high index: 
(weighted) projective spaces and complete intersection on them, and Grassmannians.

\subsection{Projective spaces}\label{proj space}
Set $h:=c_1(\cO_{\PP^n}(1))$. Then

\begin{equation*}
\ch(\PP^n)=n+\sum_{k=1}^n\frac{n+1}{k!}h^k.
\end{equation*}

In particular, $\PP^n$ is $2$-Fano, with second Chern character given by: 
$$\ch_2(\PP^n)=\frac{n+1}{2}h^2.$$

\subsection{Weighted projective spaces}\label{weighted} 
Let $\PP=\PP(a_0,\ldots, a_n)$ be a weighted projective space, with gcd$(a_0,\ldots, a_n)=1$. Denote by $H$ the effective generator of the class group $\Cl(\PP)\cong\ZZ$. Recall that $H$ is an ample $\QQ$-Cartier divisor. From the Euler sequence, 
on the smooth locus of $\PP$, we have:
$$\ch(\PP)=n+\sum_{k=1}^n\frac{a_0^k+\ldots+a_n^k}{k!}c_1(H)^k.$$

\subsection{Zero loci of sections of vector bundles}\label{zero loci of sections}
Several Fano manifolds with $\rho(X)=1$ and high index are described as the zero locus 
 $X=Z(s)\subseteq Y$ of  a global section $s$ of a vector bundle $\cE$ on a simpler variety $Y$. 
So we investigate  the 2-Fano condition in this case.

\begin{lemma}\label{Z(s)}
Let $Y$ be a smooth projective variety, and $\cE$ a vector bundle on $Y$. Let 
$s$ be a global section of $\cE$, and $X$ its zero locus $Z(s)$. 
Assume that $X$ is smooth of dimension $\dim(Y)-rk(\cE)$. Then 
$$\ch_i(X)=\big(\ch_i(Y)-\ch_i(\cE)\big)_{|X}.$$
\end{lemma}

\begin{proof}
Since the normal bundle $N_{X|Y}$ is $\cE_{|X}$, the lemma follows from the normal bundle sequence.
\end{proof}

Special cases of these are complete intersections.
If $Y$ is a smooth projective variety, 
and $X$ is a smooth complete intersection of divisors $D_1,\ldots, D_c$ in $X$, 
then Lemma~\ref{Z(s)} becomes:
\begin{equation}\label{ci's of divisors}
\ch_k(X)=\big(ch_k(Y)-\frac{1}{k!}\sum D_i^k\big)_{|X}.
\end{equation}

\subsubsection{\bf Complete intersections in $\PP^n$.}\label{ci's in proj space}\label{quadric}
Let $X$ be a smooth complete intersection of hypersurfaces of degrees 
$d_1,\ldots, d_c$.  Then by (\ref{ci's of divisors}):
$$\ch_k(X)=\frac{1}{k!}\big((n+1)-\sum d_i^k\big)h^k_{|X}.$$

It follows that 
\begin{itemize}
\item[(i) ] $X$ is $2$-Fano if an only if $\sum d_i^2\leq n$.
\item[(ii) ] $X$ is weakly $2$-Fano if an only if $\sum d_i^2\leq n+1$.
\end{itemize}

\subsubsection{\bf Complete intersections in weighted projective spaces.}\label{ci's in weighted}
We use the same notation as in  \ref{weighted}.
Let $\PP$ denote the weighted projective space $\PP(a_0,\ldots, a_n)$, with gcd$(a_0,\ldots, a_n)=1$,
and let $X$ be a smooth complete intersection of hypersurfaces with classes $d_1H,\ldots, d_cH$ in 
$\PP$. 
Assume $X$ is smooth, and contained in the smooth locus of $\PP$. 
Then the Chern character of $X$ is given by
$$\ch(X)=(n-c)+\sum_{k=1}^n\frac{a_0^k+\ldots+a_n^k-\sum d_i^k}{k!}c_1(H_{|X})^k.$$
It follows that 
\begin{itemize}
\item[(i) ] $X$ is $2$-Fano if an only if $\sum d_i^2<\sum a_i^2$.
\item[(ii) ] $X$ is weakly $2$-Fano if an only if $\sum d_i^2\leq \sum a_i^2$.
\end{itemize}

\subsection{Grassmannians}\label{G}
Let $\cS^*$ denote the dual of the universal rank $k$ vector bundle $\cS$ on  $G(k,n)$. 
The Chern classes of $\cS^*$ are given by: 
$$c_i(\cS^*)=\sigma_{1,\ldots, 1}, \quad(i\geq1)$$
where $\sigma_{a_1,\ldots, a_k}$ denotes the usual Schubert cycle on $G(k,n)$. 
Recall that $\sigma_1$ is the class of a hyperplane via the Pl\"ucker embedding
and generates $\Pic(G(k,n))$. 
Since the tangent bundle of $G(k,n)$ is given by 
$$T_{G(k,n)}=\cS^*\otimes\cQ,$$ the Chern character of $G(k,n)$ can be calculated from
$$\ch(G(k,n))=\ch(\cS^*)\ch(\cQ)=\ch(\cS^*)\big(n-\ch(\cS)\big).$$

The Chern character of $\cS^*$ is given by
$$\ch(\cS^*)=k+\sigma_1+\frac{1}{2}(\sigma_2-\sigma_{1,1})+
\frac{1}{6}(\sigma_3-\sigma_{2,1}+\sigma_{1,1,1})+\ldots.$$

As computed in \cite[2.2]{dJ-S:2fanos_1}), 
$$\ch(G(k,n))=k(n-k)+n\sigma_1+\big(\frac{n+2-2k}{2}\sigma_2-\frac{n-2-2k}{2}\sigma_{1,1}\big)+$$
$$+\frac{n-2k}{6}\big(\sigma_3-\sigma_{2,1}+\sigma_{1,1,1}\big)+\ldots.$$

The cone $\NE_2(G(k,n))$ is generated by the dual Schubert cycles 
$\sigma_2^*$ and $\sigma_{1,1}^*$.
It follows that $G(k,n)$ is $2$-Fano if and only if $n=2k$ or $2k+1$. Moreover, 
$G(k,n)$ is weakly $2$-Fano if and only if $n=2k$, $2k+1$ or $2k+2$.

\begin{rmk}\label{computations}
Complete intersections and, more generally, zero loci of vector bundles in Grassmannians will be addressed in Section~\ref{Section:ci_in homogeneous}.
We will need the following formulas, obtained by standard Chern class computations:
$$\ch(\wedge^2(\cS^*))={k\choose{2}}+(k-1)\sigma_1+\big(\frac{k-1}{2}\sigma_2-\frac{k-3}{2}\sigma_{1,1}\big)+$$
$$+\big(\frac{k-1}{6}\sigma_3-\frac{k-4}{6}\sigma_{2,1}+\frac{k-7}{6}\sigma_{1,1,1}\big)+\ldots \ \ ,$$

$$\ch(\Sym^2(\cS^*))={k+1\choose{2}}+(k+1)\sigma_1+\big(\frac{k+3}{2}\sigma_2-\frac{k+1}{2}\sigma_{1,1}\big)+$$
$$+\big(\frac{k+7}{6}\sigma_3-\frac{k+4}{6}\sigma_{2,1}+\frac{k+1}{6}\sigma_{1,1,1}\big)+\ldots \ \ .$$
\end{rmk}

%%%%%%%%%%%%%%%%%%%%%%%%%%%%%%%%%%%%%%%%%%%%%%%%%%%%%%%%%
%                                                       %
%                 SECTION 4                             %
%                                                       %
%%%%%%%%%%%%%%%%%%%%%%%%%%%%%%%%%%%%%%%%%%%%%%%%%%%%%%%%%

\section{Chern class computations}\label{Section:computations}

From the classification of  Fano manifolds with high index, we see that many of those 
with $\rho=1$ are described as double covers, while most of the ones with 
$\rho>1$ are obtained from simpler ones by blow-ups and taking projective bundles.
So in this section we compute Chern characters for these constructions.

\subsection{Double covers}

\begin{lemma}\label{double cover}
Let $f: X\ra Y$ be a finite map of degree $2$ between smooth projective varieties $X$ and $Y$. Let $R\subset X$ denote the ramification divisor, and $B=f(R)\subset Y$ the branch divisor. Then 
$f^{-1}(B)=2R$ and there is an exact sequence:
$$0\ra T_X\ra f^*T_Y\ra\cO(2R)_{|R}\ra 0.$$
The first and second Chern characters are related by
$$c_1(X)=f^*\big(c_1(Y)-\frac{1}{2}B\big),$$
$$\ch_2(X)=f^*\big(\ch_2(Y)-\frac{3}{8}B^2\big).$$
\end{lemma}

\begin{proof}
This follows from the exact sequence:
$$0\ra f^*\Omega_Y\ra\Omega_X\ra\cO(-R)_{|R}\ra0.$$
\end{proof}

\begin{cor}\label{C}
Let $f: X\ra Y$ be a finite map of degree $2$ between smooth projective varieties $X$ and $Y$. Let $B\subset Y$ be the branch divisor. Then:
\begin{itemize}
\item[(i) ]  $X$ is Fano if and only if $-K_Y-\frac{1}{2}B$ is an ample divisor. In particular, if $X$ is Fano and $B$ is nef, then $Y$ is Fano. 
\item[(ii) ]  $X$ is $2$-Fano (respectively weakly $2$-Fano) if and only if $X$ is Fano and 
$$\ch_2(Y)-\frac{3}{8}B^2>0 \quad ( \text{respectively }\geq0).$$ If $X$ is weakly $2$-Fano and $B$ is ample, then $Y$ is $2$-Fano. 
\end{itemize}
\end{cor}

\subsection{Projective Bundles}

The following two lemmas appear in 
\cite{dJ-S:2fanos_1}. 
(Note that in \cite{dJ-S:2fanos_1} the notation $\PP(E)$ stands for $\Proj(\Sym E^*)$.)
 
\begin{lemma}\label{B0}\cite[Lemma 3.1]{dJ-S:2fanos_1}
Let $X$ be a smooth irreducible projective variety and let $\cE$ be a rank $r$ vector bundle on $X$. Denote by $\pi:\PP(\cE)\ra X$ the natural projection and let $\xi=c_1(\cO_{\pi}(1))$. 
Then
$$c_1(\PP(\cE)) =\pi^*\big(c_1(X)+c_1(\cE^*)\big)+r\xi,$$ 
$$\ch_2(\PP(\cE)) =\pi^*\big(\ch_2(X)+\ch_2(\cE^*))\big)+\pi^*c_1(\cE^*)\cdot\xi+\frac{r}{2}\xi^2.$$
\end{lemma}

\begin{lemma}\cite[Prop 3.3]{dJ-S:2fanos_1}\label{B1}
Let $X$ be a smooth irreducible projective variety and let $\cE$ be a rank $2$ vector bundle on $X$. Denote by $\pi:\PP(\cE)\ra X$ the natural projection and let $\xi=c_1(\cO_{\pi}(1))$. 
Then
$$c_1(\PP(\cE)) =2\xi+\pi^*\big(c_1(X)-c_1(\cE)\big),$$ 
$$\ch_2(\PP(\cE)) =\pi^*\big(\ch_2(X) + \frac{1}{2}(c_1(\cE)^2-4c_2(\cE))\big).$$

Therefore, $\ch_2(\PP(\cE))\geq0$ if and only if 
\begin{equation}\label{proj nef}
\ch_2(X) + \frac{1}{2}(c_1(\cE)^2-4c_2(\cE))\geq0.
\end{equation}

If $\dim(X)>0$, then $\PP(\cE)$ is not $2$-Fano. $\PP(\cE)$ is weakly $2$-Fano if it is Fano and
condition (\ref{proj nef}) holds. 
\end{lemma}

\begin{cor}\label{B2}
Let $X$ be a smooth irreducible projective variety and let $L$ be a line bundle on $X$.
The projective bundle $\PP_X(\cO\oplus L)$ is not $2$-Fano and it is weakly $2$-Fano if and only if it is Fano and we have:
\begin{equation}\label{proj nef 2}
\ch_2(X)+\frac{1}{2}c_1(L)^2\geq0.
\end{equation}
In particular, (\ref{proj nef 2}) holds if $X$ is weakly $2$-Fano and $L$ is nef. 
For example:
\begin{itemize}
\item[(i)] $\PP_{\PP^n}(\cO\oplus\cO(a))$ is weakly $2$-Fano if and only if $|a|\leq n$.
\item[(ii)] $\PP_{\PP^n\times\PP^m}(\cO\oplus\cO(a,b))$ is weakly $2$-Fano if 
and only if $|a|\leq n$, $|b|\leq m$, and  $ab\geq0$.
\end{itemize}
\end{cor}

\begin{proof}
The result is an immediate consequence of Lemma \ref{B1}. For examples (i) and (ii) we use that 
 $\PP_{\PP^n}(\cO\oplus\cO(a))$ is Fano if and only if $|a|\leq n$, and similarly,
$\PP_{\PP^n\times\PP^m}(\cO\oplus\cO(a,b))$ is Fano if and only if  $|a|\leq n$ and $|b|\leq m$.
\end{proof}

\begin{example}\label{example:P(T_P)}
Consider the Fano manifold $X=\PP(T_{\PP^{n}})$, $n\geq 2$.

If $n=2$, then $X$ is not $2$-Fano, but weakly $2$-Fano by Lemma \ref{B1} since $$\ch_2(\PP^2)+\frac{1}{2}\big(c_1(\PP^2)^2-4c_2(\PP^2)\big)=
c_1(\PP^2)^2-3c_2(\PP^2)=0.$$ 

Suppose $n\geq 3$. Denote by $\pi:X\to \PP^n$ the natural morphism, and 
let $\ell\subset \PP^n$ be a line.
Consider the surface $S$ in $\pi^{-1}(\ell)$,
ruled over $\ell$,
corresponding to the surjection
$$
T_{\PP^{n}}|_{\ell}\cong \cO(2)\oplus \cO(1)^{\oplus n-1} \onto \cO(1)\oplus \cO(1).
$$
Using the formula for $\ch_2$ from Lemma~\ref{B0}, one gets that $\ch_2(X)\cdot S=-1$.
Hence $X$ is not weakly $2$-Fano.
\end{example}

\begin{lemma}\label{B3}
A product $X\times Y$ of smooth projective irreducible varieties is not $2$-Fano.
It is weakly $2$-Fano if and only if  both $X$ and $Y$ are weakly $2$-Fano.
\end{lemma}

\begin{proof}
This follows from the projection formula and the formula
$$\ch_2(X\times Y)=\pi_1^*\ch_2(X)\oplus \pi_2^*\ch_2(Y),$$
where $\pi_1:X\times Y\ra X$ and $\pi_2:X\times Y\ra Y$ are the two projections. 
For any two curves $B\subset X$ and $C\subset Y$, if we let $S=B\times C$, then 
$X\times Y$ is not $2$-Fano since  $\ch_2(X\times Y)\cdot S=0$.
\end{proof}

\subsubsection{\bf Complete intersections in in products of projective spaces.} 
\label{H_in_product}

Let $Y$ be a smooth divisor of type $(a_1,\ldots,a_r)$ in $\PP^{n_1}\times\ldots\times\PP^{n_r}$,
and set $h_i:=c_1(\pi_i^*\cO(1))$.
By a direct computation using the normal bundle sequence, we have: 
$$\ch_2(Y)=\frac{1}{2}\sum_{i=1}^r(n_i+1-a_i^2)\big(h_i^2\big)_{|Y}-
\sum_{i<j}(a_i a_j)\big(h_i\cdot h_j\big)_{|Y}.$$

We compute some examples of intersection numbers $\ch_2(Y)\cdot S$ with 
$$S={h_1}_{|Y}^{c_1}\cdot\ldots \cdot {h_r}_{|Y}^{c_r}, \quad \sum c_i=\sum n_i-3\quad (c_i\geq0).$$

Note that if $\gamma_1, \gamma_2$ are cycles on a manifold $X$ and $Y\subset X$ is a submanifold, then the intersection of the restrictions to $Y$ can be computed by
$${\gamma_1}_{|Y}\cdot{\gamma_2}_{|Y}=\gamma_1\cdot\gamma_2\cdot Y.$$

\begin{example}\label{(a,b)}
If $Y$ is a divisor of type $(a,b)$ on $\PP^n\times\PP^m$ ($a,b>0$), then 
$$\ch_2(Y)=\frac{1}{2}(n+1-a^2){h_1^2}_{|Y}+\frac{1}{2}(n+1-b^2){h_2^2}_{|Y}
-ab(h_1\cdot h_2)_{|Y}.$$ 

Since $h_1^n\cdot h_2^m=1$, it follows that 
$$\ch_2(Y)\cdot {h_1}_{|Y}^{n-2}\cdot {h_2}_{|Y}^{m-1}=\frac{b}{2}(n+1-3a^2).$$

In particular, $Y$ is not weakly $2$-Fano if either $3a^2>n+1$ or $3b^2>m+1$. 
\end{example}

\begin{example}\label{(a,b,c)}
If $Y$ is a divisor of type $(a,b,c)$ on $\PP^1\times\PP^1\times\PP^2$, then 
$$\ch_2(Y)=
\frac{1}{2}(3-c^2){h_3^2}_{|Y}-ab(h_1\cdot h_2)_{|Y}-ac(h_1\cdot h_3)_{|Y}-
bc(h_2\cdot h_3)_{|Y}.$$ 

Since $h_1^2=h_2^2=0$ and $h_1\cdot h_2\cdot h_3^2=1$, it follows that 
$$\ch_2(Y)\cdot {h_1}_{|Y}=\frac{3}{2}(1-bc^2),\quad \ch_2(Y)\cdot {h_2}_{|Y}=\frac{3}{2}(1-ac^2),$$
$$\ch_2(Y)\cdot {h_3}_{|Y}=-3abc.$$

In particular, if $a,b,c>0$ then $Y$ is not weakly $2$-Fano. 
\end{example}

\begin{example}\label{(a,b,c,d)}
If $Y$ is a divisor of type $(a_1,\ldots,a_r)$ on $(\PP^1)^r$, then 
$$\ch_2(Y)=-\sum_{i<j} a_i a_j(h_i\cdot h_j)_{|Y}.$$ 

Since $h_i^2=0$ and $h_1\cdot\ldots\cdot h_r=1$, it follows that 
$$\ch_2(Y)\cdot {h_1}_{|Y}\cdot {h_2}_{|Y}\ldots\cdot {h_{r-3}}_{|Y}=-3a_{r-2}a_{r-1}a_r.$$

In particular, if $a_i>0$ for all $i=1,\ldots r$, then $Y$ is not weakly $2$-Fano. 
\end{example}

\begin{example}\label{ci_in_product} 
Let $Y$ be a complete intersection  in 
$\PP^n\times\PP^m$ of a divisor $D_1$ of type $(a_1,b_1)$ and a divisor $D_2$ of type $(a_2,b_2)$. Then

$$\ch_2(Y)=
\frac{1}{2}(n+1-{a_1}^2-{a_2}^2){h_1^2}_{|Y}+\frac{1}{2}(m+1-{b_1}^2-{b_2}^2){h_2^2}_{|Y}-$$
$$-(a_1b_1+a_2b_2)(h_1\cdot h_2)_{|Y}.$$ 

It follows that 
$$\ch_2(Y)\cdot {h_1}_{|Y}^{n-2}\cdot {h_2}_{|Y}^{m-2}=-(a_1b_1+a_2b_2)(a_1b_2+a_2b_1)$$

In particular, $Y$ is not weakly $2$-Fano if $a_i, b_i>0$. 
\end{example}

\subsection{Blow-ups}

The following Lemma appeared first in \cite{dJ-S:2fanos_1}. See also \cite{nobili_thesis} for a detailed computation. 

\begin{lemma}\cite[Lemma 4.1]{dJ-S:2fanos_1}\label{blow-up}
Let $X$ be a smooth projective variety and let $i:Y\hra X$ be a smooth irreducible subvariety of codimension $c\geq2$. Let $f:\tX\ra X$ be the blow-up of $X$ along $Y$ and let  
$E$ be  the exceptional divisor. Denote by $j: E\hra \tX$ the natural inclusion map and let 
$\pi=f_{|E}: E\ra Y$.  Let $N$ be the normal bundle of $Y$ in $X$.
The Chern characters of $\tX$ are given by the following formulas:
$$c_1(\tX) =f^*c_1(X)-(c-1)[E],$$ 
$$\ch_2(\tX) =f^*\ch_2(X) + \frac{c + 1}{2}[E]^2-j_*\pi^*c_1(N).$$
\end{lemma}

\subsubsection{\bf Del Pezzo surfaces.}\label{del Pezzo}
Let $S_d$ ($1\leq d\leq 9$) denote a Del Pezzo surface of degree $d$, i.e., $S_d$ is
the blow-up of $\PP^2$ at $9-d$ points in general position.
By Lemma \ref{blow-up}, we have 
\begin{equation}
\ch_2(S_d)=\ch_2(\PP^2)-\frac{3}{2}(9-d)=\frac{3}{2}(d-8).
\end{equation}
It follows that the only $2$-Fano Del Pezzo surface is $\PP^2$, while $S_8=\FF_1$ and $\PP^1\times\PP^1$ (Lemma \ref{B3}) are the only other weakly $2$-Fano surfaces. 

\subsubsection{\bf The case of threefolds}\label{blow-ups 3folds}
 We compute several intersection numbers of $\ch_2(\tX)$ with surfaces in the case when 
$\tX$ is a blow-up of a threefold, first along a smooth curve (Lemma \ref{blow-up curves}), and 
then along points (Lemma \ref{blow-up points}).  

\begin{lemma}\label{blow-up curves}
Let $X$ be a smooth projective variety of dimension $3$ and let $C$ be a smooth  irreducible curve in $X$. Let $\tX$ be the blow-up of $X$ along $C$, $E$ the exceptional divisor, and $N$ the normal bundle of $C$ in $X$. Then $E^3=-\deg(N)$ and we have
$$\ch_2(\tX)\cdot E =-\frac{1}{2}\deg(N).$$

Let $T$ be a smooth surface in $X$ and let $\tT$ be  its proper transform in $\tX$. 
\begin{itemize}
\item[(i) ] If $T\cap C$ is a $0$-dimensional reduced scheme of length $r$, then 
$$\ch_2(\tX)\cdot\tT=\ch_2(X)\cdot T -\frac{3r}{2}.$$

\item[(ii) ] If $C\subset T$ and  $(C^2)_T$ denotes the self-intersection of $C$ on $T$, then
$$\ch_2(\tX)\cdot\tT=\ch_2(X)\cdot T +\frac{3}{2}(C^2)_T-\deg(N).$$
\end{itemize}
\end{lemma}

\begin{proof}
By Lemma \ref{blow-up}, we have:
$$\ch_2(\tX)\cdot E =\frac{3}{2}E^3-\big(j_*\pi^*c_1(N)\big)\cdot E.$$

Denote 
$$\xi=c_1(\cO_E(1)).$$

Since $E\cong\PP_C(N^*)$, by \cite[Rmk. 3.2.4]{Fulton}) we have
$$\xi^2+\pi^*c_1(N)\xi+\pi^*c_2(N)=0.$$ (Note that
in \cite{Fulton}, $\PP(E)=\Proj(\Sym(E^*))$.)

It follows that $\xi^2=-\deg(N)$ and hence, $$E^3=\xi^2=-\deg(N).$$  

If $\alpha$ is a cycle on $E$ and $D$ is a divisor on $X$, then 
\begin{equation}\label{formula}
j_*\alpha\cdot D=(j^*D\cdot \alpha)_E,
\end{equation}
where $( , )_E$ denotes the intersection on $E$. 
Applying (\ref{formula}) for $D=E$ and $\alpha=\pi^*c_1(N)$, it follows that 
$$\big(j_*\pi^*c_1(N)\big)\cdot E=-\big(\xi\cdot\pi^*c_1(N)\big)_E=-\deg(N),$$
$$\ch_2(T_{\tX})\cdot E = -\frac{3}{2}\deg(N)+\deg(N)=-\frac{1}{2}\deg(N).$$

For Cases (i) and (ii), by Lemma \ref{blow-up}, we have:
$$\ch_2(\tX)\cdot\tT =\ch_2(X)\cdot T+\frac{3}{2}E^2\cdot\tT-\big(j_*\pi^*c_1(N)\big)\cdot\tT.$$

Consider now Case (i). Then $\tT$ is the blow-up of $T$ along the $r$ points in  
$C\cap T$. Then $E\cap\tT$ is the union of the $r$ exceptional divisors 
of the blow-up  $\tT\ra T$.  Since $\tT_{|E}$ consists of fibers of $\pi: E\ra C$, it follows using (\ref{formula}) that $$\big(j_*\pi^*c_1(N)\big)\cdot\tT=\tT_{|E}\cdot\pi^*c_1(N)=0.$$
As $E^2\cdot\tT=\big(E_{|\tT}^2\big)_{\tT}=-r$, the result follows. 

Consider now Case (i). Then $\tT\cong T$ and $E\cap\tT$ is a section of $\pi: E\ra C$. 
By (\ref{formula}) it follows that $$\big(j_*\pi^*c_1(N)\big)\cdot\tT=\tT_{|E}\cdot\pi^*c_1(N)=\deg(N).$$
Since $E^2\cdot\tT=(C^2)_T$, the result follows. 
\end{proof}

\begin{lemma}\label{blow-up points}
Let $X$ be a smooth projective variety of dimension $3$ and let $\tX$ be the 
blow-up of $X$ along a point $q$ in $X$, with exceptional divisor $E$. Then $E^3=1$ and 
we have:
$$\ch_2(\tX)\cdot E =2.$$

Let $T$ be a surface in $X$ and let $\tT$ be  its proper transform in $\tX$. If $m\geq0$ is the multiplicity of $T$ at $q$, then we have:
$$\ch_2(\tX)\cdot\tT=\ch_2(X)\cdot T-2m.$$
\end{lemma}

\begin{proof}
Since $E_{|E}$ is the tautological line bundle $\cO_E(-1)$ on $E\cong\PP^2$, it follows that
$E^3=\big(\cO_E(-1)^2\big)_E=1$. By Lemma \ref{blow-up}, we have:
$$\ch_2(\tX)\cdot E =2E^3=2.$$

If $T$ is a surface that contains $q$ with multiplicity $m$, then 
$$\ch_2(\tX)\cdot\tT = \ch_2(\tX)\cdot (f^*T-mE)=\ch_2(X)\cdot T-2m E^3=
\ch_2(X)\cdot T-2m.$$
\end{proof}

\begin{cor}\label{blow-up curves and points}\label{D1}
Let $X$ be a smooth projective variety of dimension $3$ and $\tX$ be the blow-up of $X$ along disjoint smooth curves $C_1,\ldots, C_k$ and $l$ distinct points. If $T$ is a smooth surface in $X$ containing $0\leq s\leq l$ of the blown-up points, and intersecting $\cup_i C_i$ along a zero-dimensional reduced scheme of length $r$, then we have:
$$\ch_2(\tX)\cdot\tT=\ch_2(X)\cdot T -\frac{3r}{2}-2s.$$

In particular, $\tX$ is not weakly $2$-Fano if $$\ch_2(X)\cdot T <\frac{3r}{2}+2s.$$
\end{cor}

\begin{proof}
This is clear from Lemma \ref{blow-up curves} and Lemma \ref{blow-up points}.
\end{proof}

The results in \ref{blow-ups 3folds}  (for example Corollary \ref{D1}) give ways to check that some blow-ups of threefolds are not weakly $2$-Fano. Here we list a few more. 

\begin{cor}\label{A} (Under the assumptions and notations of Lemma \ref{blow-up curves}.) 
Suppose that either 
$g(C)>0$ and  $-K_X\cdot C>0$, or $C\cong\PP^1$ and $-K_X\cdot C>2$. Then
$$\ch_2(\tX)\cdot E<0.$$
In particular, if $X$ is Fano and either $g(C)>0$, or $X$ has index $i_X\geq3$, then 
$\ch_2(\tX)$ is not nef.

 If $C\cong \PP^1$ and $-K_X\cdot C=2$, then
$$\ch_2(\tX)\cdot E=0.$$
\end{cor}

\begin{proof}
The result follows immediately from Lemma \ref{blow-up curves} since
$$\deg(N)=\deg({T_X}_{|C})-\deg(T_C)=-K_X\cdot C-2g(C)-2.$$
\end{proof}

\begin{lemma}\label{D2} 
Let $X$ be a smooth projective threefold. Assume $X$ has a 
divisor $T$ such that $T$ is semiample and 
$$\ch_2(X)\cdot T<0.$$
Then any blow-up of $X$ along points and smooth curves is not weakly $2$-Fano.
\end{lemma}

\begin{proof}
By replacing $T$ with a multiple, we may assume that $|T|$ is a base-point free linear system.
In this case we can find a surface $T$ that avoids any of the blown-up points and intersects each of the blown-up curves in a reduced $0$-dimensional scheme of length $r\geq0$. By Lemma \ref{blow-up curves and points}, we have:
$$\ch_2(\tX)\cdot\tT=\ch_2(X)\cdot T -\frac{3r}{2}<0.$$
In particular, $\tX$ is not weakly $2$-Fano. 
\end{proof}

\begin{cor}\label{D3} 
Let $X$ be a smooth projective threefold {with~$\rho=1$}. 
If $X$ is not weakly $2$-Fano, then any blow-up of $X$ along points and smooth curves is not weakly $2$-Fano.
\end{cor}

\begin{proof}
If $X$ is not weakly $2$-Fano, it follows that there exists a surface $T\subset X$ such that 
$\ch_2(X)\cdot T<0$.
Since $T$ is an effective divisor and $X$ has Picard number $1$, it follows that $T$ is a semimaple divisor and the result follows from Lemma \ref{D2}. 
\end{proof}

\begin{cor}\label{D4} 
Let $f:X\ra Y$ be a finite map of degree $2$ between smooth projective threefolds with ample 
 branch divisor $B$. Moreover, assume $Y$ has a divisor $T$ such that $T$ is semiample and 
$$\ch_2(X)\cdot T\leq0.$$
Then any blow-up of $X$ along points and smooth curves is not weakly $2$-Fano.
\end{cor}

\begin{proof} 
By replacing $T$ with a multiple, we may assume that $|T|$ is base-point free. Note that $|f^*T|$ is also base-point free. By Lemma \ref{double cover}, we have: 
$$\ch_2(X)\cdot f^*(T)=\big(\ch_2(Y)-\frac{3}{8}B^2\big)\cdot T<0.$$
The result now follows from Lemma \ref{D2}. 
\end{proof}

%%%%%%%%%%%%%%%%%%%%%%%%%%%%%%%%%%%%%%%%%%%%%%%%%%%%%%%%%
%                                                       %
%                 SECTION 5                             %
%                                                       %
%%%%%%%%%%%%%%%%%%%%%%%%%%%%%%%%%%%%%%%%%%%%%%%%%%%%%%%%%

\section{Families of rational curves on 2-Fano manifolds}
\label{Section:rational_curves}

In this section we revise some results from \cite{AC}, to which we refer for details
and further references.

Let $X$ be a smooth complex projective uniruled variety, and $x\in X$  a general point.
There is a scheme $\rat(X,x)$ parametrizing rational curves on $X$ passing through $x$,
and it always contains 
a smooth and proper irreducible component $H_x$.
For instance, one can take $H_x$ to be an irreducible component of $\rat(X,x)$ parametrizing 
rational curves through $x$ having minimal degree with respect to some fixed
ample line bundle on $X$.
We denote by  $\pi_x:U_x\to H_x$ and $\ev_x:U_x\to X$
the usual universal family morphisms, and set $d:=\dim (H_x)$.
Since $\ev_x$ is proper and $\pi_x$ is a $\PP^1$-bundle, we have a linear map 
$$T_1={\ev_x}_*\pi_x^*:  N_1(H_x)   \to  N_{2}(X),$$
which  maps $\eff_1(H_x)\setminus \{0\}$ into $\eff_{2}(X)\setminus \{0\}$.

The variety $H_x$ 
comes with a natural polarization $L_x$, which can be defined as follows.
There is a finite morphism $\tau_x:  \ H_x  \to \PP(T_xX^{*})$ 
that sends a point parametrizing a curve smooth at $x$ to its tangent  direction at $x$.
We then set $L_x:=\tau_x^*\cO(1)$. 

The pair $(H_x, L_x)$ is called a \emph{polarized minimal family of rational curves through $x$},
and reflects much of the geometry of $X$. 
It is well understood for homogeneous spaces and complete intersections on them 
(see \cite{hwang_ICTP}).
In  \cite{AC}, we computed all the Chern classes of the variety $H_x$ in terms of the 
Chern classes of $X$ and $c_1(L_x)$. For instance, 
\begin{equation}\label{c_1 of H_x}
c_1(H_x)={\pi_x}_*\ev_x^*\big(\ch_2(X)\big)+\frac{d}{2}c_1(L_x).
\end{equation}
In particular, if $X$ is $2$-Fano (respectively weakly 2-Fano), 
then $-2K_{H_x}-dL_x$ is ample (respectively nef). 
This necessary condition is also sufficient provided that 
$T_1\big(\eff_1(H_x)\big)=\eff_2(X)$.

\begin{example}\label{lines in fibers}
We consider the special case of the Grassmannian $G(k,n)$.
The variety $H_x$ of lines in $G(k,n)$ that pass through a general point $x=[W]$ can be identified with $\PP(W)\times\PP(V/W)^*\cong  \PP^{k-1}\times\PP^{n-k-1}$,
and the map $\tau_x:  \PP^{k-1}\times\PP^{n-k-1}  \to \PP(T_xX^{*})$ is the Segre embedding. 
So the polarization $L_x$ corresponds to a divisor of type $(1,1)$.
We denote by $\pi_1$ and $\pi_2$ the projections from $\PP^{k-1}\times\PP^{n-k-1}$.
The  map 
$$
T_1:\NE_1(H_x)\ra\NE_2(G(k,n))
$$
sends classes of  lines in the fibers of $\pi_1$ and $\pi_2$, to the dual cycles
$\sigma_2^*$ and $\sigma_{1,1}^*$, respectively.

Now let $X=H_1\cap\ldots\cap H_c\subseteq G(k,n)$ be a smooth complete intersection
of hyperplane sections $H_1,\ldots, H_c$ under the Pl\"ucker embedding, with $c\leq n-2$.
We may assume that $x\in X$ is a general point, and consider the variety of lines in $X$,
$Z_x\subset H_x$. Notice that $Z_x$ is 
a complete intersection  of $c$ divisors $D_i$ of type $(1,1)$ in $H_x$:
$$Z_x=D_1\cap\ldots\cap D_c\subset H_x \cong \PP^{k-1}\times\PP^{n-k-1}.$$

\medskip

\noindent {\bf Claim.}
\begin{itemize}
\item[(i) ] If $c\leq k-1$, then $Z_x$ contains a line from a fiber of $\pi_2$. In particular, 
$X$ contains a surface with class $\sigma_{1,1}^{*}$.

\item[(ii) ]  If $c<n-k-1$, then $Z_x$ contains a line from a fiber of $\pi_1$. In particular, 
$X$ contains a surface with class $\sigma_{2}^{*}$.
\end{itemize}

In particular, if $c\leq k-1$ and $n>2k$, then  the natural map 
$$u_2:\NE_2(X)\ra\NE_2(G(k,n))$$
is surjective.

\begin{proof}
Let $x_0,\ldots, x_{k-1}$ (respectively $y_0,\ldots, y_{n-k-1}$) denote the coordinates on $\PP^{k-1}$ (respectively on $\PP^{n-k-1}$). 
Each divisor $D_i$ has an equation of type:
$$x_0F_0^{(i)}+\ldots+x_{k-1}F_{k-1}^{(i)}=0,$$
where $F_j^{(i)}$ are linear forms in $(y_i)$. Clearly, if $c<k-1$, then $Z_x$ contains a line from any fiber of $\pi_2$. By the same argument, if $c<n-k-1$, then $Z_x$ contains a line from any fiber of $\pi_1$, and this proves (ii).

Note that if $c=k-1$ and $k\leq n-k-1$, then the locus in $\PP^{n-k-1}$ where the $k$ minors of size $(k-1)\times (k-1)$ of the matrix of linear forms $\big(F_j^{(i)}\big)$ vanish is non-empty. This proves (i) and the claim follows. 
\end{proof}

Note that inequalities (i) and (ii) are optimal.
Indeed, 
consider the case when  $X=H_1\cap H_2\subset G(2,5)$. Let $x_0, x_1$ 
(respectively $y_0,y_1, y_2$) denote the coordinates on $\PP^1$ (respectively on $\PP^2$). 
The variety $Z_x$ is a complete intersection in $\PP^1\times\PP^2$ of two divisors of type $(1,1)$:
$$D_1: x_0F_0+x_1F_1=0,$$
$$D_2: x_0G_0+x_1G_1=0,$$
where $F_i, G_i$ are linear forms in $(y_i)$. 
Thus $Z_x$ is isomorphic to the smooth conic
$F_0G_1-F_1G_0=0$
in $\PP^2$,  and 
$Z_x\subset\PP^1\times\PP^2$ 
is a curve of type $(2,2)$. It follows that $T_1:\NE_1(Z_x)\ra\NE_2(X)$ 
maps the fundamental class of
$Z_x$ to $\sigma_2^*+\sigma_{1,1}^*$. 
\end{example}

%%%%%%%%%%%%%%%%%%%%%%%%%%%%%%%%%%%%%%%%%%%%%%%%%%%%%%%%%
%                                                       %
%                 SECTION 6                              %
%                                                       %
%%%%%%%%%%%%%%%%%%%%%%%%%%%%%%%%%%%%%%%%%%%%%%%%%%%%%%%%%

\section{Complete intersections on homogeneous spaces}\label{Section:ci_in homogeneous}

\subsection{Complete intersections in Grassmannians}
We apply the results in Section \ref{zero loci of sections} to the case when the ambient space is a Grassmannian.

\begin{prop}\label{ci's in G}
Consider a smooth complete intersection 
$$X=(d_1H)\cap\ldots\cap (d_cH)\subseteq G(k,n) 
\quad (2\leq k\leq\frac{n}{2}, 1\leq c),$$ 
where $H=\sigma_1$ is the class of a hyperplane class via the Pl\"ucker embedding.

The Chern character of $X$ is given by:
$$\ch(X)=\big(k(n-k)-c\big)+(n-\sum d_i)\sigma_1+$$
$$\big(\frac{n+2-2k-\sum d_i^2}{2}\sigma_2-\frac{n-2-2k+\sum d_i^2}{2}\sigma_{1,1}\big)+$$
$$+\big(\frac{n-2k-\sum d_i^3}{6}\big(\sigma_3+\sigma_{1,1,1}\big)-
\frac{n-2k+\sum d_i^3}{6}\sigma_{2,1}\big)+\ldots.$$ 

Then $X$ is Fano if and only if $\sum d_i<n$. 
Moreover,  $X$ is not weakly $2$-Fano if 
$$\sum d_i^2\geq n-2k+2,$$  
with the exception of the case when $n=2k, c=2, d_1=d_2=1$, in which case
$X$ is weakly $2$-Fano (see also Proposition~\ref{linear ci's in G}).
\end{prop}

\begin{proof}[Proof of Proposition~\ref{ci's in G}]
The formula for $\ch(X)$ follows from the formula for $\ch(G(k,n))$ computed in \ref{G} 
and the formula for complete intersections from \ref{zero loci of sections}.
The criterion for $X$ to be Fano follows immediately. 

To prove the last statement, we may assume $X$ is Fano.
If $\sum d_i^2\geq n-2k+2$ then $\ch_2(X)=a\sigma_2+b\sigma_{1,1}$, with
$a, b\leq0$. Note that $b=0$ if and only if $n=2k$ and $\sum d_i^2=2$, i.e., $c=2$ and $d_1=d_2=1$. In this case, $\ch_2(X)=0$. But if $b<0$, then either $a<0$, in which case 
$\ch_2(X)\cdot S<0$ for any surface $S\subset X$, or $a=0$ and we have: 
$$\sum d_i^2=n-2k+2, \quad \ch_2(X)=-(n-2k){\sigma_{1,1}}_{|X}\quad (n>2k).$$

Since $\sigma_{1,1}\cdot \sigma_1^{\dim G(k,n)-2}>0$ and ${\sigma_1}_{|X}$ is ample, 
in the latter case $X$ is not weakly $2$-Fano for
\begin{equation}\label{inters}
\ch_2(X)\cdot \sigma_1^{\dim(X)-2}<0.
\end{equation}
\end{proof}

\begin{prop}\label{linear ci's in G}
Consider a smooth complete intersection 
$$X=H_1\cap\ldots\cap H_c\subseteq G(k,n) 
\quad (2\leq k\leq\frac{n}{2}, 1\leq c<n),$$ 
of hyperplane sections $H_1,\ldots, H_c$ under the Pl\"ucker embedding.
Then 
$$\ch_2(X)=\frac{n+2-2k-c}{2}\sigma_2-\frac{n-2-2k+c}{2}\sigma_{1,1}.$$
Moreover:
\begin{itemize}
\item[(i) ] If $c\geq n-2k+2$ (with $c\neq2$ if $n=2k$) then $X$ is not weakly $2$-Fano.
\item[(ii) ] If $c\leq k-1$ and $n\geq2k+2$, then $X$ is not weakly $2$-Fano.
\item[(iii) ] If $n=2k$ then $X$ is (weakly) $2$-Fano if and only if $c=1$ ($c=1,2$).
\item[(iv) ] If $n=2k+1$, then $X$ is not $2$-Fano; $X$ is weakly $2$-Fano if and only if $c=1$, 
and possibly when $X=H_1\cap H_2\subset G(2,5)$. 
\end{itemize}
\end{prop}

\begin{proof}
By Proposition \ref{ci's in G}, if $c\geq n-2k+2$, and if we are not in the case when $n=2k$ and $c=2$, then $X$ is not weakly $2$-Fano. This gives (i). 
For (ii), note that if $c\leq k-1$ then, by  Example~\ref{lines in fibers}, 
the natural map $u_2:\NE_2(X)\ra\NE_2(G(k,n))$ is surjective.
If $n-2k-2+c>0$ and $n>2k$, then part (ii) follows, since 
$$\ch_2(X)\cdot\sigma_{1,1}^*=-\frac{n-2k-2+c}{2}<0.$$

Part (iii) follows immediately, since if $n=2k$ then 
$$\ch_2(X)=\frac{2-c}{2}\big(\sigma_2+\sigma_{1,1}\big)=\frac{2-c}{2}\sigma_1^2.$$

We now prove (iv). Assume that $n=2k+1$. We have: 
$$\ch_2(X)=\frac{3-c}{2}\sigma_2+\frac{1-c}{2}\sigma_{1,1}.$$ 

By (i), if $c\geq3$ then $X$ is not weakly $2$-Fano. Assume that $c\leq2$. If $k\geq3$, then
$c\leq2\leq k-1$. By Example~\ref{lines in fibers}, the natural map 
$u_2:\NE_2(X)\ra\NE_2(G(k,n))$ is surjective. 
It follows that,
in this case, $X$ is weakly $2$-Fano if and only if the coefficients of $\sigma_2$ and $\sigma_{1,1}$ in the formula for $\ch_2(X)$ are non-negative, i.e., $c=1$.
Note that $X$ is not $2$-Fano in this case, as $\ch_2(X)\cdot\sigma_{1,1}^*=0$.  

We are left to analyze what happens in the case when $k<3$, i.e., the case of $G(2,5)$. If $c=1$ then $\ch_2(X)=\sigma_2\geq0$, and $X$ is weakly $2$-Fano.  By 
Example~\ref{lines in fibers}, the natural map 
$u_2:\NE_2(X)\ra\NE_2(G(2,5))$ is surjective,
and $X$ is not $2$-Fano since
$\ch_2(X)\cdot\sigma_{1,1}^*=0$. Now assume $c=2$. Then we have:
$$\ch_2(X)=\frac{1}{2}\sigma_2-\frac{1}{2}\sigma_{1,1}.$$ 
By Example~\ref{lines in fibers}, $X$ contains a surface $S$ with class $\sigma_2^*+\sigma_{1,1}^*$.
Clearly, $X$ is not $2$-Fano, since $\ch_2(X)\cdot S=0$.
\end{proof}

\subsection{Orthogonal Grassmannians}\label{OG(k,n)}
We fix $Q$ a nondegenerate symmetric bilinear form on the $n$-dimensional vector 
space $V$.
Let $OG(k,n)$ be the subvariety of the Grassmannian $G(k,n)$ parametrizing linear subspaces that are isotropic with respect to $Q$. 
 
If $n\neq 2k$ then  $OG(k,n)$ is a Fano manifold of dimension $\frac{k(2n-3k-1)}{2}$ and $\rho=1$. 
On the other hand, $OG(k,2k)$ has two connected components \cite[p. 737]{Griffiths-Harris}:
If  $\Sigma\subset V$ is a fixed isotropic subspace of dimension $k$ in $V$, then one component $OG_+(k,2k)$, corresponds to $[W]\in OG(k,2k)$ such that 
$\dim(W\cap \Sigma)\equiv k$ (mod $2$), while the other component $OG_-(k,2k)$ corresponds to those $[W]\in OG(k,2k)$ such that $\dim(W\cap \Sigma)\not\equiv k$ (mod $2$). The two components are disjoint and isomorphic. Note also that $$OG(k-1,2k-1)\cong OG_+(k,2k).$$ 

The orthogonal Grassmannian $OG(k,n)$ is the zero locus in $G(k,n)$ of a global section of the vector bundle $\Sym^2(\cS^*)$. Using this description and the formula for $\ch\big(G(k,n)\big)$ described in \ref{G}, standard Chern class computations show that for any component $X$ of 
$OG(k,n)$ we have:

$$\ch(X)=\frac{k(2n-3k-1)}{2}+(n-k-1)\sigma_1+$$
$$+\big(\frac{n-3k-1}{2}\sigma_2-\frac{n-3k-3}{2}\sigma_{1,1}\big)+$$
$$+\big(\frac{n-3k-7}{6}\sigma_3-\frac{n-3k-4}{6}\sigma_{2,1}+
\frac{n-3k-1}{6}\sigma_{1,1,1}\big)+\ldots.$$

\subsubsection{\bf Complete intersections in $OG_+(k,2k)$.}\label{ci's in OG_+}
Our main reference in what follows is \cite{Coskun}. We consider now one component 
$OG_+(k,2k)$ of the orthogonal Grassmannian $OG(k,2k)$. For the reader's convenience, we recall the description of Schubert varieties in $OG_+(k,2k)$. Let 
$$F_1\subset F_2\subset\ldots\subset F_k$$ be an isotropic flag in $V$, with $[F_k]\in OG_+(k,2k)$.  
This induces a second flag 
$$F_{k-1}\subset F_{k-1}^{\perp}\subset F_{k-2}^{\perp}\subset\ldots\subset F_1^{\perp}\subset V.$$
Here, by abuse of notation, we denote by $F_{k-1}^{\perp}$  an isotropic subspace of dimension $k$
parametrized by $OG_-(k,2k)$ and such that
$F_{k-1}\subset F_{k-1}^{\perp}$.

For each decreasing sequence
$$\lambda: k-1\geq\lambda_1>\lambda_2>\ldots>\lambda_s\geq0 \quad (s\leq k),$$
(where we assume $k-s$  is even)  we denote by 
$$\mu: k-1\geq\mu_{s+1}>\mu_{s+2}>\ldots>\mu_k\geq0$$
the sequence obtained by removing $k-1-\lambda_i$ from $k-1,\ldots,0$.
For each sequence $\lambda$ as above, we have a Schubert variety of codimension $\sum\lambda_i$:
$$\Omega^0_{\lambda}=\{[W]\in OG(k,2k) \ | \ \dim(W\cap F_{k-\lambda_i})= i, \ 
 \dim(W\cap F_{\mu_j}^{\perp})= j \}.$$
Let  $\Omega_{\lambda}$ be the closure of  $\Omega^0_{\lambda}$ and denote by
$\tau_{\lambda}$ its cohomology class.
The cohomology of $OG_+(k,2k)$ is generated by the classes $\tau_{\lambda}$.
In particular,
$b_4(OG_+(k,2k))=1$. 

\begin{claim}\label{relation}
On $OG_+(k,2k)$ we have
$\sigma_2=\sigma_{1,1}=\frac{1}{2}\sigma_1^2$. 
\end{claim}

\begin{proof}

Since $b_4=1$, it is enough to find a surface $S$ in $OG_+(k,2k)$  such that $\sigma_2\cdot S=\sigma_{1,1}\cdot S$. Let 
$S=\Omega_{k-1, k-2,\ldots,3,1}$
(the unique Schubert variety of dimension $2$).  One can show 
that $\sigma_2\cdot S=\sigma_{1,1}\cdot S=2$. 
We leave this fun computation to the reader. 
\end{proof}

\begin{prop}
$OG_+(k,2k)$ is a $2$-Fano manifold. 

Consider a smooth complete intersection 
$$X=(d_1H)\cap\ldots\cap (d_cH)\subseteq OG_+(k,2k)\quad (k\geq3),$$
where $H=\frac{1}{2}\sigma_1$ denotes a hyperplane section of the 
half-spinor embedding of $OG_+(k,2k)$. The Chern character of $X$ is given by:
$$\ch(X)=\frac{k(k-1)}{2}+\big(2k-2-\sum d_i\big)H+\big(\frac{4-\sum d_i^2}{2}\big)H^2+\ldots.$$

Then $X$ is Fano if and only if $\sum d_i<2k-2$. Moreover, $X$ is $2$-Fano if and only if 
all $d_i=1$ and $c\leq3$.  The only other cases when $X$ is weakly $2$-Fano are when 
$c=4$, $d_1=\ldots=d_4=1$ and $c=2$, $d_1=d_2=2$. 

\end{prop}

\begin{proof}
Since $\sigma_1^2=\sigma_2+\sigma_{1,1}$, by Claim \ref{relation}, we obtain 
$$\ch_2(OG_+(k,2k))=\frac{1}{2}\sigma_1^2.$$

In particular, $OG_+(k,2k)$ is $2$-Fano. Recall that a hyperplane section of $OG_+(k,2k)$ via the Pl\"ucker embedding is linearly equivalent to $2H$, where  $H$ is a hyperplane section of the spinor embedding \cite[Proposition 1.7]{Mukai95}. It follows that $2H=\sigma_1$. The result now follows from the formula for the  Chern character of $OG(k,n)$. 
\end{proof}

\subsection{Symplectic Grassmannians}\label{SG}

We fix $\omega$ a non-degenerate antisymmetric bilinear form on the $n$-dimensional vector 
space $V$, $n$ even.
Let $SG(k,n)$ be the subvariety of the 
Grassmannian $G(k,n)$ parametrizing linear subspaces that are isotropic with respect to $\omega$. Then $SG(k,n)$ is a Fano manifold of dimension $\frac{k(2n-3k+1)}{2}$ and $\rho(X)=1$. Notice that $X$ is the zero locus in $G(k,n)$ of a global section of the vector bundle $\wedge^2(\cS^*)$.
Using this description and the formula for $\ch\big(G(k,n)\big)$ described in \ref{G}, 
standard Chern class computations show that
$$\ch(SG(k,n))=\frac{k(2n-3k+1)}{2}+(n-k+1)\sigma_1+$$
$$+\big(\frac{n-3k+3}{2}\sigma_2-\frac{n-3k+1}{2}\sigma_{1,1}\big)+$$
$$+\big(\frac{n-3k+1}{6}\sigma_3-\frac{n-3k+4}{6}\sigma_{2,1}+
\frac{n-3k+7}{6}\sigma_{1,1,1}\big)+\ldots.$$

\subsubsection{\bf Complete intersections in $SG(k,2k)$.}\label{ci's in SG(k,2k)} 
The symplectic Grassmannian 
$SG(k,2k)$ is a Fano manifold with $b_4=1$. For example, note that  $b_4\big(SG(k,2k)\big)=b_4\big(OG(k,2k+1)\big)$ (see for instance
\cite[Section 3.1]{BS}), and $b_4\big(OG(k,2k+1)\big)=1$ 
(see Section \ref{ci's in OG_+}).  

\begin{claim} On $SG(k,2k)$ we have
$\sigma_2=\sigma_{1,1}=\frac{1}{2}\sigma_1^2$.
\end{claim}

\begin{proof}
Since $\sigma_1^2=\sigma_2+\sigma_{1,1}$, it is enough to prove that on $SG(k,2k)$ we have $\sigma_2=\sigma_{1,1}$. Since $b_4(SG(k,2k))=1$, we are done if we find a surface $S$ in $SG(k,2k)$ such that $S\cdot\sigma_2=S\cdot\sigma_{1,1}$. Let $x$ be a general point on 
$SG(k,2k)$ and let $H_x$ denote the space of lines on $SG(k,2k)$ that pass through $x$.  
Recall from \cite[5.5]{AC}) that  
$$H_x\cong\PP^{k-1}\subset\PP^{k-1}\times\PP^{k-1}$$ 
is the diagnoal embedding. Let $S$ be the surface in $SG(k,2k)$ corresponding to a line 
in $H_x\cong\PP^{k-1}$ via the map $T_1:\NE_1(H_x)\ra\NE_2(SG(k,2k))$. 
It follows that the class of $S$ is 
$\sigma_2^*+\sigma_{1,1}^*$.  Clearly, $S\cdot\sigma_2=S\cdot\sigma_{1,1}=1$.
\end{proof}

It follows from \ref{SG} that
$$\ch(SG(k,2k))=\big(\frac{k(k+1)}{2}\big)+(k+1)\sigma_1+\frac{1}{2}\sigma_1^2+\ldots$$
In particular, $SG(k,2k)$ is $2$-Fano (as proved also in \cite[5.5]{AC}) and we have the following consequence:

\begin{prop}\label{ci's in SG}
Consider a smooth complete intersection 
$$X=(d_1H)\cap\ldots\cap (d_cH)\subseteq SG(k,2k)\quad (k\geq2, c\geq1),$$
where $H=\sigma_1$ is a hyperplane section under the Plucker embedding. 

The Chern character of $X$ is given by:
$$\ch(X)=\big(\frac{k(k+1)}{2}-r\big)+(k+1-\sum d_i)\sigma_1+
\frac{1-\sum d_i^2}{2}\sigma_1^2+\ldots$$

Then $X$ is Fano if and only if $\sum d_i<k$. Moreover, 
$X$ is weakly $2$-Fano if and only if $c=d_1=1$. In this case $X$ is not $2$-Fano. 
\end{prop}

\subsection{Complete intersections in homogeneous spaces $G_2/P_2$}\label{ci's in G_2/P_2}

If $G$ is a group of type $G_2$, there exist two maximal parabolic subgroups $P_1$ and $P_2$ in $G$. The quotient variety $G/P_1$ is isomorphic to a $5$-dimensional quadric $Q\subset\PP^6$, and $G/P_2$ is a Mukai variety of genus $g=10$ (see Theorem \ref{Mukai varieties}):
$$G_2/P_2\subset\PP^{13}.$$
One has $b_4(G_2/P_2)=1$ (see for instance \cite[Proposition 4.5 and Appendix A.3]{Anderson}),
and $G_2/P_2$ is $2$-Fano by \cite[5.7.]{AC}.

Recall from \cite[1.4.5]{hwang_ICTP} 
the polarized minimal family of rational curves through 
a general point $y \in G_2/P_2$ is
$$(H_y,L_y)\cong(\PP^1,\cO(3)).$$ 
Let $H$ denote the hyperplane class (the generator of the Picard group). 
We claim that 
$$\ch_2(G_2/P_2)=\frac{1}{2}H^2.$$
This will follow from the following more general remark, applied to $Y=G_2/P_2$.

\begin{rmk}[Complete intersections in varieties with $\rho(Y)=b_4(Y)=1$]\label{E}
Let $Y$ be a Fano manifold with $\rho(X)=1$, and 
$H$ an ample generator of $\Pic(Y)$.
Let $y\in Y$ be a general point, and $(H_y,L_y)$ a 
polarized minimal family of rational curves through $y$, as defined in 
Section~\ref{Section:rational_curves}.
Suppose that $\dim(H_y)\geq1$.
If $b_4(Y)=1$, then the map $$T_1: \NE_1(H_y)\ra\NE_2(Y)$$ is clearly surjective. 
Let 
$C\subset H_y$ be a complete curve, and $S=T_1([C])$ the corresponding surface class on $Y$. 
By \eqref{c_1 of H_x}, the second Chern character of $Y$ is given by: 
$$\ch_2(Y)=aH^2,\quad a\in\frac{1}{2}\ZZ, \quad 
a(H^2\cdot S)=-(K_{H_y}-\frac{d}{2}L_y)\cdot C.$$ 
In particular, 
$a\leq  -(K_{H_y}-\frac{d}{2}L_y)\cdot C$.

Now consider a complete intersection:
$$X=(d_1H)\cap\ldots \cap(d_cH)\subset Y.
$$
The natural map $u_2:\NE_2(X)\ra\NE_2(Y)$
is surjective. 
Thus, by \eqref{ci's of divisors},  $X$ is $2$-Fano (respectively weakly 2-Fano)  if and only if 
it is Fano and $\sum d_i^2<2a$ (respectively $\leq 2a$).
\end{rmk}

We have the following consequence: 

\begin{prop}\label{linear sections in G_2/P_2}
A linear section $H_1\subset G_2/P_2$ is weakly $2$-Fano, but not $2$-Fano. 
A linear section $H_1\cap H_2\subset G_2/P_2$ is not $2$-Fano.  
\end{prop}

%%%%%%%%%%%%%%%%%%%%%%%%%%%%%%%%%%%%%%%%%%%%%%%%%%%%%%%%%%%%%

% SECTION 7

%%%%%%%%%%%%%%%%%%%%%%%%%%%%%%%%%%%%%%%%%%%%%%%%%%%%%%%%%%%%%

\section{Fano manifolds with high index and $\rho=1$}\label{Section:2-Fano_classification_rho=1}  

In this section we address $n$-dimensional Fano manifolds $X$ with index $i_X\geq n-2$ and 
$\rho(X)=1$. We also treat those with bigger Picard number for $n>4$.

Recall from Section~\ref{Section:first_examples} that $\PP^n$ and $Q^n\subset\PP^{n+1}$ 
are $2$-Fano for $n\geq3$.

\subsection{Del Pezzo manifolds}\label{check del Pezzos}\label{check index2}
We go through the classification in Theorem~\ref{del Pezzo varieties}.  
We first consider manifolds with $\rho=1$. 

\subsubsection{\bf Degree $d=5$.} 
We saw in Section~\ref{G} that the Grassmannian $G(2,5)$ is $2$-Fano. Consider now a linear section 
$$
X=H_1\cap\ldots\cap H_c\subset G(2,5)\quad (c\geq1).
$$ 
By Proposition \ref{linear ci's in G}(iv), if $c=3$, the threefold $X$ is not weakly $2$-Fano. 
If $c=1$, then $X$ is weakly $2$-Fano, but not $2$-Fano. If $c=2$, then $X$ is not $2$-Fano. 
We could not decide if in this case $X$ is weakly $2$-Fano. We raise the following: 

\begin{question}\label{ques1}
Is a linear section $\PP^7\cap G(2,5)\subset\PP^9$ weakly $2$-Fano?
\end{question}

\subsubsection{\bf Degree $d=4$.} By \ref{ci's in proj space}, a Del Pezzo variety of type $Y_4$ is $2$-Fano if and only if $n\geq6$ and weakly $2$-Fano if and only if $n\geq5$.

\subsubsection{\bf Degree $d=3$.} By \ref{ci's in proj space},  a  Del Pezzo variety of type $Y_3$ is 
$2$-Fano if and only if $n\geq8$ and weakly $2$-Fano if and only if $n\geq7$. 

\subsubsection{\bf Degree $d=2$.} By Corollary \ref{C}(ii) or \ref{ci's in weighted}, 
Del Pezzo varieties of type $Y_2$ are $2$-Fano (respectively weakly $2$-Fano)  
if and only if $n>11$ (respectively  $n\geq 11$). 

\subsubsection{\bf Degree $d=1$.} By Corollary \ref{ci's in weighted}, Del Pezzo varieties of type $Y_1$ are (weakly) $2$-Fano if and only if  $n>23$ ($n\geq 23$).

\subsubsection{\bf Del Pezzo manifolds with $\rho>1$.} 
All the Del Pezzo manifolds with $\rho>1$ are weakly $2$-Fano but not $2$-Fano.
For $\PP^2\times\PP^2$ and $\PP^1\times\PP^1\times\PP^1$ this follows from 
Lemma \ref{B3}, for $\PP(\cO_{\PP^2}\oplus\cO_{\PP^2}(1))$ from 
Corollary~\ref{B2}, and for  $\PP(T_{\PP^2})$ from
Example~\ref{example:P(T_P)}.

\

\begin{rmk} 
We get the following classification of weakly $2$-Fano Del Pezzo manifolds:
\begin{itemize}
\item All the Del Pezzo manifolds with $\rho>1$ are weakly $2$-Fano. 
\item The only Del Pezzo manifolds with $\rho=1$ that are weakly $2$-Fano are:
\begin{itemize}
\item[($d=5$)]  $G(2,5)$ and its linear sections of codimension $1$ (and possibly codimension $2$, see Question \ref{ques1});
\item[($d=4$)]  Complete intersections of quadrics $Q\cap Q'\subset\PP^{n+2}$ if $n\geq5$;
\item[($d=3$)]  Cubic hypersurfaces $Y_3\subset\PP^{n+1}$ if $n\geq7$;
\item[($d=2$)]  Degree $4$ hypersurfaces in $\PP(2,1,\ldots,1)$ if $n\geq11$; 
\item[($d=1$)]  Degree $6$ hypersurfaces in $\PP(3,2,1,\ldots,1)$ if $n\geq23$. 
\end{itemize}
\end{itemize}
\end{rmk}

\subsection{Mukai manifolds}\label{check Mukai manifolds} 

\subsubsection{\bf Mukai manifolds of dimension $>4$ and  $\rho>1$.}

Recall that the only Mukai manifolds of dimension $>4$ and  $\rho>1$ are $\PP^3\times \PP^3$,
$\PP^2 \times Q^3$, $\PP (T_{\PP^3})$ and $\PP_{\PP^3}(\cO(1)\oplus \cO)$.
The manifolds $\PP^3\times \PP^3$ and $\PP^2 \times Q^3$ are weakly $2$-Fano but not $2$-Fano
by Lemma~\ref{B3}, while $\PP (T_{\PP^3})$ and $\PP_{\PP^3}(\cO(1)\oplus \cO)$ are not weakly Fano 
by Example~\ref{example:P(T_P)} and Corollary~\ref{B2}, respectively.

\

Next we go through the classification of Mukai manifolds with $\rho=1$ in Theorem \ref{Mukai varieties}.  

\subsubsection{\bf Genus $g\leq5$.}
Consider the case of complete intersections. By \ref{ci's in proj space}: 

\begin{itemize}
\item $X_4\subset\PP^{n+1}$ is  $2$-Fano (respectively weakly $2$-Fano)  if and only if $n\geq15$ 
(respectively  $n\geq14$). 
\item $X_{2\cdot3}\subset\PP^{n+2}$ is $2$-Fano (respectively weakly $2$-Fano)   if and only if 
$n\geq11$  (respectively  $n\geq10$) 
\item $X_{2\cdot2\cdot2}\subset\PP^{n+3}$ is $2$-Fano (respectively weakly $2$-Fano)   if and only if 
$n\geq9$  (respectively $n\geq8$).
\end{itemize}

Consider the case of double covers. Using Corollary \ref{C}, we have:  
\begin{itemize}
\item[(i) ]  A double cover $X\ra\PP^n$ branched a long a sextic is $2$-Fano 
(respectively weakly $2$-Fano) if and only if $n\geq27$ (respectively $n\geq26$). 
\item[(ii) ] A double cover $X\ra Q\subset\PP^{n+1}$ branched a long the intersection of the quadric $Q$ with a quartic hypersurafce, is  $2$-Fano (respectively weakly $2$-Fano) if and only if $n\geq15$ 
(respectively $n\geq14$). 
\end{itemize}

% COMPUTATIONS
%Indeed, in Case (i) by Corollary \ref{C}, $X$ is (weakly) $2$-Fano if and only if 
%$$\ch_2(\PP^n)-\frac{3}{8}(6H)^2=\frac{n+1}{2}H^2-\frac{27}{2}H^2=
%\frac{n-26}{2}H^2>0\quad (\geq0).$$
%In Case (ii) by Corollary \ref{C}, $X$ is (weakly) $2$-Fano if and only if 
%$$\ch_2(Q)-\frac{3}{8}(4H_{|Q})^2=\frac{n-2}{2}H_{|Q}^2-6H_{|Q}^2=
%\frac{n-14}{2}H_{|Q}^2>0\quad (\geq0).$$

\subsubsection{\bf Genus $6$. } A linear section $X$ of $\Sigma^6_{10}$ is 
isomorphic to one of the following (\cite[Proposition 5.2.7]{IP}): 
\begin{itemize}
\item[(i) ] A complete intersection in $G(2,5)$ of a linear subspace and a quadric.
\item[(ii) ] A double cover of a smooth linear section $Y$ of $G(2,5)$, 
branched along a quadric section $B$ of $Y$.
\end{itemize}

In Case (i), it follows from Proposition \ref{ci's in G}
that $X$ is not weakly $2$-Fano. Consider now Case (ii). Set $c:=\codim(Y)$. 
Then $X$ is not weakly $2$-Fano by Corollary \ref{C}, since we have:  
$$\ch_2(Y)-\frac{3}{8}B^2=\frac{3-c}{2}\sigma_2+\frac{1-c}{2}\sigma_{1,1}-
\frac{3}{8}(2\sigma_1)^2=-\frac{c}{2}\sigma_2-\frac{c+2}{2}\sigma_{1,1},$$
$$\big(\ch_2(Y)-\frac{3}{8}B^2\big)\cdot{\sigma_1}_{|Y}^{\dim(Y)-2}=
\big(-\frac{c}{2}\sigma_2-\frac{c+2}{2}\sigma_{1,1}\big)\cdot\sigma_1^4<0.$$

\subsubsection{\bf Genus $7$. }  
By \ref{ci's in OG_+}, the manifold $OG_+(5,10)$ is $2$-Fano and a linear section of codimension $c$ is $2$-Fano (respectively weakly $2$-Fano) if and only if $c<4$ (respectively $c\leq4$).

\subsubsection{\bf Genus $8$. }

We saw in Section~\ref{G} that  the Grassmannian $G(2,6)$ is weakly $2$-Fano. Let $X\subset G(2,6)$ be a linear section of codimension $c$. By Proposition \ref{linear ci's in G}, if $c\geq4$ or $c=1$, then $X$ is not weakly $2$-Fano. Assume that $c=3$. Then 
$$\ch_2(X)=\frac{1}{2}(\sigma_2-3\sigma_{1,1}).$$
By a straightforward calculation, $\sigma_1^6=9\sigma_2^*+5\sigma_{1,1}^*$. It follows that
$$\ch_2(X)\cdot{\sigma_1}_{|X}^3=\ch_2(X)\cdot\sigma_1^6=-3<0.$$
In particular, $X$ is not weakly $2$-Fano. 

Assume now that $c=2$. Then $\ch_2(X)=\sigma_2-\sigma_{1,1}$. 
By Example~\ref{lines in fibers}, the variety $H_x\subset\PP^1\times\PP^3$
defined in Section~\ref{Section:rational_curves} is isomorphic to a smooth quadric surface in $\PP^3$ (via the projection $\pi_2$). 
The map $T_1:\NE_1(H_x)\ra\NE_2(G(k,n))$ sends the classes of the lines in the two rulings of the quadric surface $H_x$  to the classes $\sigma_2^*$ and $\sigma_2^*+\sigma_{1,1}^*$. In particular, $X$ is not $2$-Fano, as $\ch_2(X)\cdot(\sigma_2^*+\sigma_{1,1}^*)=0$. We could not decide if in this case $X$ is weakly $2$-Fano. We raise the following: 

\begin{question}\label{ques2}
Is a linear section $\PP^{12}\cap G(2,6)\subset\PP^{14}$ weakly $2$-Fano?
\end{question}

\subsubsection{\bf Genus $9$. } 

We saw in Section~\ref{SG} that 
the symplectic Grassmannian $SG(k,2k)$ is $2$-Fano.
By \ref{ci's in SG(k,2k)}, a codimension $c\geq1$ linear section $X$ of $SG(k,2k)$ is not $2$-Fano. The only other case when $X$ is weakly $2$-Fano  is for $c=1$. 

\subsubsection{\bf Genus $10$. } 

We saw in Section~\ref{ci's in G_2/P_2} that 
the variety $G_2/P_2$ is $2$-Fano. 
By Proposition \ref{linear sections in G_2/P_2}, 
a codimension $c\geq1$ linear section in $G_2/P_2$ is not $2$-Fano and it is weakly $2$-Fano
if and only if $c=1$.

\subsubsection{\bf Genus $g=12$} \label{g=12}
By Remark~\ref{computations},
$$c_1(\wedge^2(\cS^*))=2\sigma_1,\quad \ch_2(\wedge^2(\cS^*))=\sigma_2.$$
It follows from Lemma \ref{Z(s)} and the computation of $\ch(G(3,7))$ made in Section~\ref{G} that
the Chern characters of the threefold $X=X_{22}$ are given by: 
$$c_1(X)=\big(c_1\big(G(3,7)\big)-3 c_1\big(\wedge^2(\cS^*)\big)\big)_{|X}={\sigma_1}_{|X}$$
$$\ch_2(X)=\big(\ch_2\big(G(3,7)\big)-3\ch_2\big(\wedge^2(\cS^*)\big)\big)_{|X}=
-\frac{3}{2}{\sigma_2}_{|X}+\frac{1}{2}{\sigma_{1,1}}_{|X}.$$
Since $\rho=1$ and $\dim(X)=3$, $b_4(X)=1$. In particular, the restrictions
${\sigma_2}_{|X}$ and ${\sigma_{1,1}}_{|X}$ are  multiples of the positive codimension $2$-cycle $$A:=(\sigma_1^2)_{|X}.$$
We claim  that 
$${\sigma_{1,1}}_{|X}=\frac{5}{11}A,\quad  {\sigma_2}_{|X}=\frac{6}{11}A.$$ 
To see this, it is enough to prove that 
$$\big(\sigma_2\cdot\sigma_1)_X=12,\quad \big(\sigma_{1,1}\cdot\sigma_1)_X=10,$$
where $(,)_X$ denotes the intersection on $X$. Since $X$ is the zero locus of a global section of the rank $9$ vector bundle $\cE=(\wedge^2(\cS^*))^{\oplus 3}$, it follows that
$$\big(\sigma_2\cdot\sigma_1)_X=\sigma_2\cdot\sigma_1\cdot c_9(\cE).$$
By a standard computation with Chern classes, 
$$c_9(\cE)=c_3(\wedge^2(\cS^*))^3=\big(c_1(\cS^*)c_2(\cS^*)-c_3(\cS^*)\big)^3=\sigma_{2,1}^3.$$ 
It is a straightforward exercise in Schubert calculus to check that
$$\sigma_{2,1}^3=4\sigma_{4,4,1}+8\sigma_{4,3,2}+2\sigma_{3,3,3}.$$
It follows that
$$\sigma_1\cdot\sigma_{2,1}^3=12\sigma_2^*+10\sigma_{1,1}^*.$$
Then $\ch_2(X)\cdot A=-13<0$. Hence, $X_{22}$ is not weakly $2$-Fano. 

\

\begin{rmk} 
We get the following classification of weakly $2$-Fano Mukai manifolds with $\rho=1$:
\begin{itemize}
\item[(1)] Complete interesection in projective spaces:

\begin{itemize}
\item[($g=3$)]  Degree $4$ hypersurafces in $\PP^{n+1}$ if $n\geq15$; 
\item[($g=4$)]  Complete intersections $X_{2\cdot3}\subset\PP^{n+2}$ if $n\geq11$;
\item[($g=5$)]  Complete intersections $X_{2\cdot2\cdot2}\subset\PP^{n+3}$ if $n\geq9$.
\end{itemize}

\

\item[(2)]  Complete interesection in weighted projective spaces:
\begin{itemize}
\item[($g=2$)]  Degree $6$ hypersurafces in $\PP(3,1,\ldots,1)$ if $n\geq26$; 
\item[($g=3$)] Complete intersections of two quadrics in $\PP(2,1,\ldots,1)$, $n\geq14$. 
\end{itemize}

\

\item[(3)]  With genus $g\geq6$: 
\begin{itemize}
\item[($g=7$)]  $OG_+(5,10)$ and linear sections of codimension $c\leq4$; 
\item[($g=8$)]  $G(2,6)$ and possibly a linear section of codimension $2$ in $G(2,6)$ (see Question \ref{ques2});
\item[($g=9$)] $SG(3,6)$ and linear sections of codimension $1$; 
\item[($g=10$)] $G_2/P_2$ and linear sections of codimension $1$. 
\end{itemize}

\end{itemize}

\end{rmk}

%%%%%%%%%%%%%%%%%%%%%%%%%%%%%%%%%%%%%%%%%%%%%%%%%%%%%%%%%
%                                                       %
%                 SECTION 8                             %
%                                                       %
%%%%%%%%%%%%%%%%%%%%%%%%%%%%%%%%%%%%%%%%%%%%%%%%%%%%%%%%%

\section{Fano threefolds with Picard number $\rho\geq2$}\label{>1}\label{Section:classification_rho>1}

By the results of Mori-Mukai \cite{Mori-Mukai} (see also \cite{Mori-Mukai erratum}) there are $88$ types of Fano threefolds with Picard number $\rho(X)\geq2$, up to deformation. We will go through the list in \cite{Mori-Mukai} and check that none of them is $2$-Fano. We point out those that are weakly $2$-Fano. We recall the terminology and notation from \cite{Mori-Mukai}: 

\

(i) $V_d$ ($1\leq d\leq 5$) denotes a Fano $3$-fold of index $2$, with $\rho(X)=1$ and 
degree $d$ (See Thm. \ref{del Pezzo varieties}). 

\

(ii) $W$ is a smooth divisor of $\PP^2\times\PP^2$ of bidegree $(1,1)$. It is isomorphic to the $\PP^1$-bundle $\PP(T_{\PP^2})$ over $\PP^2$, and appears as (32) in the following list.

\

(iii) The blow-up of $\PP^3$ at a point is denoted by $V_7$. It
appears as (35) in the following list.
The smooth quadric in $\PP^4$ is denoted by $Q$. 

\

(iv) $S_d$ ($1\leq d\leq 7$) is a Del Pezzo surface of degree $d$. 
$\FF_1$ is the blow-up of $\PP^2$ at a point. 

\

(v) All curves are understood to be smooth and irreducible, and all intersections are understood to be scheme theoretic. 

\

(vi) A divisor $D$ (respectively a curve $C$) on the product variety $$M=\PP^{n_1}\times\ldots\times\PP^{n_m}$$ is of multi-degree $(a_1,\ldots, a_m)$ if $\cO_M(D)\cong\otimes_{i=1}^m \pi_i^*\cO_{\PP^{n_i}}(a_i)$ (respectively if 
$$C\cdot \pi_i^*\cO_{\PP^{n_i}}(a_i)=a_i$$ for all $i=1,\ldots, m$), where $\pi_i$ is the 
projection of $M$ onto the $i$-th factor.

\subsection{Fano $3$-folds with $\rho=2$}\label{Picard number 2} 
We go through the list in  \cite[Table 2]{Mori-Mukai} and check that each Fano $3$-fold in the list is not $2$-Fano. We point out the cases in which the $3$-fold is weakly $2$-Fano. 

\

({\bf 1}) The blow-up of $V_1$ with center an elliptic curve which is an intersection of two members of $\big| -\frac{1}{2} K_{V_1}\big|$. This is not weakly $2$-Fano by  Corollary \ref{A}. 

\

({\bf 2}) A double cover of $\PP^1\times\PP^2$ whose branch locus is a divisor of bidegree $(2,4)$. Since $\PP^1\times\PP^2$  is not $2$-Fano, this is not weakly $2$-Fano by 
Corollary \ref{C}(ii). 

\

({\bf 3})  The blow-up of $V_2$ with center an elliptic curve which is an intersection of two members of $\big| -\frac{1}{2} K_{V_2}\big|$. This is not weakly $2$-Fano by  Corollary \ref{A}. 

\

({\bf 4}) The blow-up of $\PP^3$ with center an intersection of two cubics. Since $\PP^3$ has index $4$, this is not weakly $2$-Fano by  Corollary \ref{A}. 

\

({\bf 5}) The blow-up of $V_3\subset\PP^4$ with center a plane cubic in it. This is not weakly $2$-Fano by  Corollary \ref{A}. 

\

({\bf 6a}) A divisor on $\PP^2\times\PP^2$ of bidegree $(2,2)$ is not weakly $2$-Fano by 
Example~\ref{(a,b)}.

\

({\bf 6b})  A double cover of $W$ whose branch locus is a member of $|-K_W|$. Since $W=\PP(T_{\PP^2})$ is not $2$-Fano by Lemma \ref{B1}, its double cover
 is not weakly $2$-Fano by Corollary \ref{C}(ii). 

\

({\bf 7}) The blow-up of $Q\subset \PP^4$ with center an intersection of two members of 
$\big| \cO_Q(2) \big|$. Since $Q$ is a Fano threefold with index $3$,  
this is not weakly $2$-Fano by  Corollary \ref{A}. 

\

({\bf 8})  A double cover of $V_7$ whose branch locus is a member $B$ of $|-K_{V_7}|$ such that 
either (a) $B\cap D$ is smooth, or (b) $B\cap D$ is reduced but not smooth, where $D$ is the exceptional divisor of the blow-up $V_7\ra\PP^3$. Since $V_7=\PP_{\PP^2}(\cO\oplus\cO(1))$ 
is not $2$-Fano by Lemma \ref{B1}, this $3$-fold is not weakly $2$-Fano by Corollary \ref{C}(ii).  

\

({\bf 9}) The blow-up of $\PP^3$ with center a curve of degree $7$ and genus $5$ which is an intersection of cubics.  This is not weakly $2$-Fano by 
Corollary \ref{A}. 

\

({\bf 10}) The blow-up of $V_4\subset\PP^5$ with center an elliptic curve which is an intersection of two hyperplane sections.  This is not weakly $2$-Fano by  Corollary \ref{A}. 
 
\

({\bf 11}) The blow-up of $V_3\subset\PP^4$ with center a line on it. Since $V_3$ has Picard number $1$ and is not weakly $2$-Fano (see Section~\ref{check index2}), this   is not weakly $2$-Fano by Corollary \ref{D3}. 

\

({\bf 12}) The blow-up of $\PP^3$ with center a curve of degree $6$ and genus $3$ which is an intersection of cubics.  This is not weakly $2$-Fano by  Corollary \ref{A}. 

\

({\bf 13}) The blow-up of $Q\subset \PP^4$ with center a curve of degree $6$ and genus $2$.  This is not weakly $2$-Fano by  Corollary \ref{A}. 

\

({\bf 14}) The blow-up of $V_5\subset\PP^6$ with center an elliptic curve which is an intersection of two hyperplane sections.  This is not weakly $2$-Fano by  Corollary \ref{A}. 

\

({\bf 15}) The blow-up of $\PP^3$ with center an intersection of a quadric $A$ and a cubic $B$ such that either (a) $A$ is smooth, or (b) $A$ is reduced, but not smooth. This is not weakly $2$-Fano by  Corollary \ref{A}.

\

({\bf 16}) The blow-up of $V_4\subset\PP^5$ with center a conic on it. Since $V_4$ has $\rho=1$ and is not weakly $2$-Fano (see Section~\ref{check index2}), this is not weakly $2$-Fano by  Corollary \ref{D3}. 

\

({\bf 17}) The blow-up of $Q\subset \PP^4$ with center an elliptic curve of degree $5$ on it.  This is not weakly $2$-Fano by  Corollary \ref{A}. 

\

({\bf 18}) A double cover of $\PP^1\times\PP^2$ whose branch locus is a divisor of bidegree $(2,2)$. Since $\PP^1\times\PP^2$ is not $2$-Fano by Lemma \ref{B3}, this is not weakly $2$-Fano by Corollary \ref{C}(ii). 

\

({\bf 19}) The blow-up of $V_4\subset\PP^5$ with center a line on it. Since $V_4$ has $\rho=1$ and is not weakly $2$-Fano (see Section~\ref{check index2}), this is not weakly $2$-Fano by  Corollary \ref{D3}. 

\

({\bf 20}) The blow-up of $V_5\subset\PP^6$ with center a twisted cubic on it. Since $V_5$ has $\rho=1$ and is not weakly $2$-Fano (see Section~\ref{check index2}), this is not weakly $2$-Fano by  Corollary \ref{D3}. 

\

({\bf 21}) The blow-up of $Q\subset \PP^4$ with center a twisted quartic (i.e., a smooth rational curve of degree $4$ which spans $\PP^4$) on it. Since $Q$ is a Fano threefold with index $3$,  
this is not weakly $2$-Fano by  Corollary \ref{A}. 

\

({\bf 22}) The blow-up of $V_5\subset\PP^6$ with center a twisted conic on it. Since $V_5$ has $\rho=1$ and is not weakly $2$-Fano (see Section~\ref{check index2}), this is not weakly $2$-Fano by  Corollary \ref{D3}. 

\

({\bf 23}) The blow-up of $Q\subset\PP^4$ with center an intersection of $A\in|\cO_Q(1)|$ and $B\in|\cO_Q(2)|$ such that either (a) $A$ is smooth or, (2) $A$ is not smooth. Since $Q$ has index $3$, this is not weakly $2$-Fano by  Corollary \ref{A}. 

\

({\bf 24}) A divisor on $\PP^2\times\PP^2$ of bidegree $(1,2)$ is not weakly $2$-Fano by 
Example~\ref{(a,b)}. 

\

({\bf 25}) The blow-up of $\PP^3$ with center an elliptic curve which is an intersection of two quadrics.  Since $\PP^3$ has index $4$, this is not weakly $2$-Fano by  Corollary \ref{A}. 

\

({\bf 26}) The blow-up of $V_5\subset\PP^6$ with center a line on it. Since $V_5$ has $\rho=1$ and is not weakly $2$-Fano (see Section~\ref{check index2}), this is not weakly $2$-Fano by  Corollary \ref{D3}. 

\

({\bf 27}) The blow-up of $\PP^3$ with center a twisted cubic.  Since $\PP^3$ has index $4$, this is not weakly $2$-Fano by  Corollary \ref{A}. 

\

({\bf 28}) The blow-up of $\PP^3$ with center a plane cubic.  This is not weakly $2$-Fano by  Corollary \ref{A}. 

\

({\bf 29}) The blow-up of $Q\subset \PP^4$ with center a conic on it.  Since $Q$ has index $3$, this is not weakly $2$-Fano by  Corollary \ref{A}. 

\

({\bf 30}) The blow-up of $\PP^3$ with center a conic.  Since $\PP^3$ has index $4$, this is not weakly $2$-Fano by  Corollary \ref{A}. 

\

({\bf 31}) The blow-up of $Q\subset \PP^4$ with center a line on it.  Since $Q$ has index $3$, this is not weakly $2$-Fano by  Corollary \ref{A}. 

\

({\bf 32}) $W\cong \PP(T_{\PP^2})$. This is not $2$-Fano, but weakly $2$-Fano by 
Example~\ref{example:P(T_P)}.

\

({\bf 33}) The blow-up of $\PP^3$ with center a line.  Since $\PP^3$ has index $4$, this is not weakly $2$-Fano by  Corollary \ref{A}. 

\

({\bf 34}) The product $\PP^1\times\PP^2$ is not $2$-Fano, but weakly $2$-Fano by Lemma \ref{B3}. 

\

({\bf 35}) $V_7\cong \PP_{\PP^2}(\cO\oplus\cO(1))$. This is not $2$-Fano, but weakly $2$-Fano by Corollary \ref{B2}. 

\

({\bf 36}) The blow-up of the Veronese cone $W_4\subseteq\PP^6$ with center the vertex, that is
the $\PP^1$-bundle $\PP(\cO\oplus\cO(2))$ over $\PP^2$. This is not $2$-Fano, but weakly $2$-Fano by Corollary \ref{B2}. 

\

\subsection{Fano $3$-folds with $\rho=3$} We go through the list in  \cite[Table 3]{Mori-Mukai} and check that each Fano $3$-fold in the list is not $2$-Fano. We point out the cases in which the $3$-fold is weakly $2$-Fano. 

\

({\bf 1}) A double cover of $\PP^1\times\PP^1\times\PP^1$ whose branch locus is a divisor of tridegree $(2,2,2)$. Since  $\PP^1\times\PP^1\times\PP^1$  is not $2$-Fano by Lemma \ref{B3}, this is not weakly $2$-Fano by Corollary \ref{C}(ii).

\

({\bf 2}) A member $X$ of the linear system $|\cO_{\pi}(1)^{\otimes 2}\otimes\cO(2,3)|$ on the 
$\PP^2$-bundle $\PP(\cO\oplus\cO(-1,-1)^{\oplus 2})$ over $\PP^1\times\PP^1$ such that 
$X\cap Y$ is irreducible. Here  $Y$ is a member of $|\cO_{\pi}(1)|$. 

We prove that $X$ is not weakly $2$-Fano by a direct computation. 
Set $\cE:=\cO\oplus\cO(-1,-1)^{\oplus 2}$ and let 
$\pi:\PP(\cE)\ra\PP^1\times\PP^1$ be the corresponding projection map. 
If $\pi_1, \pi_2$ are the two projections from $\PP^1\times\PP^1$, we let $H_i=\pi_i^*\cO(1)$ ($i=1,2$). Let $\xi=c_1(\cO_{\pi}(1))$. 
By Lemma \ref{B0} and the formula (\ref{ci's of divisors}) we have
$$\ch_2(\PP(\cE))=6\pi^*(H_1\times H_2)+2\pi^*(H_1+H_2)\cdot\xi+\frac{3}{2}\xi^2,$$
$$\ch_2(X)=\big(2\pi^*(-H_1-2H_2)\cdot\xi-\frac{1}{2}\xi^2\big)_{|X}.$$

We claim that $\ch_2(X)\cdot(\pi^*H_1)_{|X}<0$. This is a direct computation:
$$\ch_2(X)\cdot(\pi^*H_1)_{X}=
\ch_2(X)\cdot\pi^*H_1\cdot\big(\pi^*(2H_1+3H_2)+2\xi\big)=-\frac{15}{2}.$$

\

({\bf 3}) A divisor on $\PP^1\times\PP^1\times\PP^2$ of tridegree $(1,1,2)$ is not weakly $2$-Fano by
Example~\ref{(a,b,c)}.

\

({\bf 4})  The blow-up of $Y$ (No. (18) in the list for $\rho=2$) with center a smooth fiber of 
$Y\ra\PP^1\times\PP^2\ra\PP^2$. Recall that $Y\ra\PP^1\times\PP^2\ra\PP^2$ is a double cover branched along a divisor of bidegree $(2,2)$. Apply Corollary \ref{D4} to deduce that this is not weakly $2$-Fano. 

\

({\bf 5}) The blow-up of $\PP^1\times\PP^2$ with center a curve $C$ of bidegree $(5,2)$ such that the composition $C\hra \PP^1\times\PP^2\ra\PP^2$ is an embedding. Since 
$-K_{\PP^1\times\PP^2}\cdot C=16>2$, this is not weakly $2$-Fano by  Corollary \ref{A}. 

\

({\bf 6}) The blow-up of $\PP^3$ with center a disjoint union of a line and an elliptic curve of degree $4$.  This is not weakly $2$-Fano by  Corollary \ref{A}. 

\

({\bf 7}) The blow-up of $W$ with center an elliptic curve of degree $4$.  This is not weakly $2$-Fano by  Corollary \ref{A}. 

\

({\bf 8}) A member $X$ of the linear system $|\pi_1^*g^*\cO(1)\otimes\pi_2^*\cO(2)|$ on $\FF_1\times\PP^2$, where $\pi_i$ ($i=1,2$) is the projection to the $i$-th factor and 
$g: \FF_1\ra\PP^2$ is the blow-up map. We prove that $X$ is not weakly $2$-Fano by a direct computation. Set $h_1:=c_1\big(\pi_1^*g^*\cO(1)\big)$ and $h_2:=c_1\big(\pi_2^*\cO(1)\big)$. By (\ref{del Pezzo}), $\ch_2(\FF_1)=0$. 
By (\ref{ci's of divisors}), we have:

$$\ch_2(X)=\big(\ch_2(\FF_1\times\PP^2)-\frac{1}{2}X^2\big)_{|X}=
\big(\frac{3}{2}h_2^2-\frac{1}{2}(h_1+2h_2)^2\big)_{|X}=$$
$$=\big(-\frac{1}{2}h_1^2-\frac{1}{2}h_2^2-2h_1h_2\big)_{|X}.$$

We claim that $\ch_2(X)\cdot {h_2}_{|X}<0$. This is a direct computation: 
$$\ch_2(X)\cdot {h_2}_{|X}=\ch_2(X)\cdot h_2\cdot X=
\ch_2(X)\cdot h_2\cdot(h_1+2h_2)=-3.$$

\

({\bf 9})  The blow-up of the cone $W_4\subset\PP^6$ over the Veronese surface $R_4\subset\PP^5$ with center a disjoint union of the vertex and a quartic curve in $R_4\cong\PP^2$.  Since 
the center of the blow-up is a curve of genus $3$, 
this is not weakly $2$-Fano by  Corollary \ref{A}. 

\

({\bf 10}) The blow-up of $Q\subset\PP^4$ with center a disjoint union of two conics on it.  Since $Q$ has index $3$, this is not weakly $2$-Fano by  Corollary \ref{A}. 

\

({\bf 11}) The blow-up of $V_7$ with center an elliptic curve which is an intersection of two members of of $|-\frac{1}{2}K_{V_7}|$.  This is not weakly $2$-Fano by  Corollary \ref{A}. 

\

({\bf 12}) The blow-up of $\PP^3$ with center a disjoint union of a line and a twisted cubic.
Since $\PP^3$ has index $4$, this is not weakly $2$-Fano by  Corollary \ref{A}.

\

({\bf 13})  The blow-up of $W\subset\PP^2\times\PP^2$ with center a curve $C$ of bidegree $(2,2)$ on it such that the composition of $C\hra W\hra\PP^2\times\PP^2$ with the projection $\pi_i:\PP^2\times\PP^2\ra\PP^2$ is an embedding for both $i=1,2$. 
Since $-K_W\cdot C=8$, this is not weakly $2$-Fano by  Corollary \ref{A}. 

\

({\bf 14}) The blow-up of $\PP^3$ with center a union of a cubic in a plane $S$ and a point not in $S$. This is not weakly $2$-Fano by  Corollary \ref{A}. 

\

({\bf 15}) The blow-up of $Q\subset\PP^4$ with center a disjoint union of a line and a conic on it. Since $Q$ has index $3$, this is not weakly $2$-Fano by  Corollary \ref{A}. 

\

({\bf 16}) The blow-up of $V_7$ with center the strict transform of a twisted cubic passing through the center of the blow-up $V_7\ra\PP^3$. Since $-K_{V_7}\cdot C=10$, this is not weakly $2$-Fano by  Corollary \ref{A}. 

\

({\bf 17}) A smooth divisor on $\PP^1\times\PP^1\times\PP^2$ of tridegree $(1,1,1)$ is not weakly $2$-Fano by Example~\ref{(a,b,c)}.

\

({\bf 18}) The blow-up of $\PP^3$ with center a disjoint union of a line and a conic. Since $\PP^3$ has index $4$, this is not weakly $2$-Fano by  Corollary \ref{A}. 

\

({\bf 19}) The blow-up $X$ of $Q\subset\PP^4$ with center two points $p$ and $q$ on it which are not collinear. By \ref{quadric}, $$\ch_2(Q)=\frac{1}{2}h^2_{|Q}.$$ 

It $\tT$ is the proper transform of a general hyperplane section $T$ of $Q$ that passes through $p$, by Lemma \ref{D1}, we have
$$\ch_2(X)\cdot\tT=\ch_2(Q)\cdot T- 2=\frac{1}{2}h^3-2=-1.$$
In particular, $X$ is not weakly $2$-Fano. 

\begin{rmk}\label{Table 3- No(19)}
Moreover, note that $\tT$ is a base-point free divisor on $X$. It follows from Corollary \ref{D2} that no blow-up of $X$ along points and disjoint smooth curves is weakly $2$-Fano. 
\end{rmk}

\

({\bf 20}) The blow-up of $Q\subset\PP^4$ with center two disjoint lines on it. Since $Q$ has index $3$, this is not weakly $2$-Fano by  Corollary \ref{A}. 

\

({\bf 21}) The blow-up of $\PP^1\times\PP^2$ with center a curve $C$ of bidegree $(2,1)$. 
Since $-K_{\PP^1\times\PP^2}\cdot C=7$, this is not weakly $2$-Fano by  Corollary \ref{A}. 

\

({\bf 22})  The blow-up of $\PP^1\times\PP^2$ with center a conic $C$ in $\{t\}\times\PP^2$ ($t\in\PP^1$). Since $-K_{\PP^1\times\PP^2}\cdot C=6$, this is not weakly $2$-Fano by  Corollary \ref{A}. 

\

({\bf 23}) The blow-up of $V_7$ with center a conic $C$ passing through the center of the blow-up $V_7\ra\PP^3$. Recall that $V_7$ is the blow-up of $\PP^3$ at a point. Since $-K_{V_7}\cdot C=6$, this is not weakly $2$-Fano by  Corollary \ref{A}. 

\

({\bf 24}) The fiber product $X=W\times_{\PP^2}\FF_1$ where $W\ra\PP^2$ is the $\PP^1$-bundle $\PP(T_{\PP^2})$ and $\pi: \FF_1\ra\PP^2$ is the blow-up map. This is not weakly 
$2$-Fano: Since $X=\PP_{\FF_1}(\pi^*T_{\PP^2})$, by Lemma \ref{B1}, $\ch_2(X)\geq0$ if and only if 
$$\ch_2(\FF_1)+\frac{1}{2}\pi^*\big(c_1(\PP^2)^2-4c_2(\PP^2)\big)\geq0.$$

By Lemma \ref{blow-up}, $\ch_2(\FF_1)=\pi^*\ch_2(\PP^2)+\frac{3}{2}E^2$, where $E$ is the exceptional divisor of $\FF_1$. Hence, $X$ is not weakly $2$-Fano, since 
$$\pi^*\big(c_1(\PP^2)^2-3c_2(\PP^2)\big)+\frac{3}{2}E^2=-\frac{3}{2}<0.$$

\

({\bf 25}) The blow-up of $\PP^3$ with center two disjoint lines, that is, 
$\PP(\cO(1,0)\oplus\cO(0,1))$ over $\PP^1\times\PP^1$. This is not weakly $2$-Fano by 
Corollary \ref{B2}. 

\

({\bf 26}) The blow-up of $\PP^3$ with center a disjoint union of a point and a line. Since $\PP^3$ has index $4$, this is not weakly $2$-Fano by  Corollary \ref{A}. 

\

({\bf 27}) $\PP^1\times\PP^1\times\PP^1$ is weakly $2$-Fano and not $2$-Fano by Lemma \ref{B3}.

\

({\bf 28}) $\PP^1\times\FF_1$ is weakly $2$-Fano by Lemma \ref{B3}. 

\

({\bf 29}) The blow-up $X$ of $V_7$ with center a line $L$ on the exceptional divisor $E\cong\PP^2$ of the blow-up $\pi:V_7\ra\PP^3$ at a point $p$. The line $L$ corresponds to a plane $\Lambda\subset\PP^3$ passing through $p$. By Lemma \ref{blow-up}, we have:
$$\ch_2(V_7)=2(\pi^*h)^2+2E^2.$$

Let $T$ be a plane through the point $p$, not containing the plane $\Lambda$. The proper transform $\tT$ of $T$ in $X$ intersects $L$ in a point. By Lemma \ref{D1},
$$\ch_2(\tX)\cdot\tT=\ch_2(V_7)\cdot T-\frac{3}{2}=(2(\pi^*h)^2+2E^2)\cdot(\pi^*h-E)-
\frac{3}{2}=-\frac{3}{2}.$$

In particular, $\tX$ is not weakly $2$-Fano.

\

({\bf 30}) The blow-up $X$ of $V_7$ along the proper transform of a line passing through the center of the blow-up $V_7\ra\PP^3$. By Lemma \ref{blow-up curves}, we have
$$\ch_2(X)\cdot E=0,$$ where $E$ is the exceptional divisor corresponding to the line. In particular, $X$ is not $2$-Fano. This is the only case of a Fano threefold where we could not decide if $X$ is weakly $2$-Fano.

\

({\bf 31}) The blow-up of the cone over a smooth quadric surface in $\PP^3$ with center the vertex, that is, the $\PP^1$-bundle $\PP(\cO\oplus\cO(1,1))$ over $\PP^1\times\PP^1$. 
This is weakly $2$-Fano and not $2$-Fano by Corollary \ref{B2}.

\

\subsection{Fano $3$-folds with $\rho=4$} We go through the list in  \cite[Table 4]{Mori-Mukai}, \cite{Mori-Mukai erratum} and check that each Fano $3$-fold in the list is not $2$-Fano. 
We point out the cases in which the $3$-fold is weakly $2$-Fano. 

\

({\bf 1}) A smooth divisor on $\PP^1\times\PP^1\times\PP^1\times\PP^1$ of tridegree $(1,1,1,1)$ is not weakly $2$-Fano by  Example~\ref{(a,b,c,d)}.

\

({\bf 2}) The blow-up of the cone over a smooth quadric surface $S\subset\PP^3$ with center a disjoint union of the vertex and an elliptic curve on $S$.  This is not weakly $2$-Fano by  Corollary \ref{A}.

\

({\bf 3}) The blow-up of $\PP^1\times\PP^1\times\PP^1$ with center a curve of tridegree $(1,1,2)$. Since $-K_{\PP^1\times\PP^1\times\PP^1}\cdot C=8$, this is not weakly $2$-Fano by  Corollary \ref{A}. 

\

({\bf 4})  The blow-up of $Y$ (No. (19) in the list for $\rho=3$) with center the strict transform of a conic on $Q$ passing through $p$ and $q$. This is not $2$-Fano by Remark \ref{Table 3- No(19)}. 

\

({\bf 5})  The blow-up of $\PP^1\times\PP^2$ with center two disjoint curves $C_1$ and $C_2$ of bidegree $(2,1)$ and $(1,0)$ respectively. Since $-K_{\PP^1\times\PP^2}\cdot C_1=7$, this is not weakly $2$-Fano by  Corollary \ref{A}. 

\

({\bf 6})  The blow-up of $\PP^3$ with center three disjoint lines, that is, the blow-up of $\PP^1\times\PP^1\times\PP^1$ with center the tridiagonal curve. Since $\PP^3$ has index $4$, this is not weakly $2$-Fano by  Corollary \ref{A}.

\

({\bf 7})  The blow-up of $W\subset\PP^2\times\PP^2$ with center two disjoint curves 
$C_1$ and $C_2$ of bidegree $(0,1)$ and $(1,0)$. Since $-K_W\cdot C_i=3$, this is not weakly $2$-Fano by  Corollary \ref{A}.

\

({\bf 8})  The blow-up of $\PP^1\times\PP^1\times\PP^1$ with center a curve $C$ of tridegree $(0,1,1)$. Since $-K_{\PP^1\times\PP^1\times\PP^1}\cdot C=4$, this is not weakly $2$-Fano by  Corollary \ref{A}. 

\

({\bf 9})  The blow-up $X$ of $Y$ (No. (25) in the list for $\rho=3$) with center an exceptional line of the blowing up $Y\ra\PP^3$. Recall that $Y$ is the blow-up of $\PP^3$ along two disjoint lines. If $T$ is the proper transform in $Y$ of a plane in $\PP^3$ intersecting the two lines at general points, then by Corollary \ref{D1} we have: 
$$\ch_2(Y)\cdot T=\ch_2(\PP^3)\cdot H-3=-1.$$

Since $T$ is disjoint from the exceptional line blown-up, 
$$\ch_2(X)\cdot\tT= \ch_2(Y)\cdot T=-1.$$

In particular, $X$ is not weakly $2$-Fano.  

\

({\bf 10})  $\PP^1\times S_7$ is not weakly $2$-Fano by  \ref{del Pezzo} and Lemma \ref{B3}.

\

({\bf 11})  The blow-up $X$ of $\PP^1\times\FF_1$ with center $\{t\}\times C$, where $t\in T$ and $C$ is the exceptional curve of the first kind on $\FF_1$. If $\FF_1\ra\PP^2$ is the blow-up of a point $p\in\PP^2$, let $T$ be the surface $\PP^1\times L$, where $L$ is the proper transform of a general line trough the point $p$. Since $T$ intersects $\{t\}\times C$ in one point, it follows from Corollary \ref{D1}, that 
$$\ch_2(\tX)\cdot\tT=\ch_2(\PP^1\times\FF_1)\cdot T - \frac{3}{2}=- \frac{3}{2}.$$

In particular, $X$ is not weakly $2$-Fano.  

\

({\bf 12})  The blow-up $X$ of $Y$ (No. (33) in the list for $\rho=2$) with center two exceptional lines of the blowing-up $Y\ra\PP^3$ along a line $L$. Let $T$ be the proper transform on $Y$ of a plane in $\PP^3$ that contains $L$. It follows from Corollary \ref{D1} that
$$\ch_2(Y)\cdot T=\ch_2(\PP^3)\cdot H+\frac{3}{2} \big(L^2\big)_T-\deg(N_{L|\PP^3})=
\frac{3}{2}.$$

Let $\tT$ be the proper transform of $T$ in $X$. Since $\tT$ intersects the blown-up curves in two points, it follows by Corollary \ref{D1} that 
$$\ch_2(\tX)\cdot\tT=\ch_2(Y)\cdot T-3=-\frac{3}{2}.$$

In particular, $X$ is not weakly $2$-Fano.

\

({\bf 13}) (See \cite{Mori-Mukai erratum}.) The blow-up of $\PP^1\times\PP^1\times\PP^1$ with center a curve of tridegree $(1,1,3)$. Since $-K_{\PP^1\times\PP^1\times\PP^1}\cdot C=10$, this is not weakly $2$-Fano by  Corollary \ref{A}. 

\

\subsection{Fano $3$-folds with $\rho\geq5$} We go through the list in  \cite[Table 3]{Mori-Mukai} and check that each Fano $3$-fold in the list is not $2$-Fano. We point out the cases in which the $3$-fold is weakly $2$-Fano. 

\

({\bf 1})  The blow-up $X$ of $Y$ (No. (29) in the list for $\rho=2$) with center three 
exceptional lines of the blowing-up $Y\ra Q$ along a conic $C$. 

Let $T$ be the proper transform on $Y$ of a general hyperplane section of $Q$. Note that $T$ will intersect $C$ in two points. It follows from Corollary \ref{D1} that
$$\ch_2(Y)\cdot T=\ch_2(Q)\cdot H-3=-2.$$

Since $T$ is disjoint from the three exceptional lines of the blow-up, it follows that $\ch_2(\tX)\cdot\tT=\ch_2(Y)\cdot T=-2$.  In particular, $X$ is not weakly $2$-Fano.

\

({\bf 2})  The blow-up $X$ of $Y$ (No. (25) in the list for $\rho=3$) with center two exceptional lines $l, l'$ of the blowing-up $$\phi: Y\ra\PP^3$$ such that $l, l'$ lie on the same irreducible component of the exceptional set of $\phi$. Recall that  $\phi$ is the blow-up of two disjoint lines $L_1$ and $L_2$ in $\PP^3$.

Let $T$ be the proper transform on $Y$ of a general plane in $\PP^3$ that intersects both $L_1$ and $L_2$. It follows from Corollary \ref{D1} that
$$\ch_2(Y)\cdot T=\ch_2(\PP^3)\cdot H-3=-1.$$

Since $T$ is disjoint from $l$ and $l'$, 
$\ch_2(\tX)\cdot\tT=\ch_2(Y)\cdot T=-1$.
In particular, $X$ is not weakly $2$-Fano.

\

({\bf 3}) $\PP^1\times S_6$ is not weakly $2$-Fano by  \ref{del Pezzo} and Lemma \ref{B3}.

\

({\bf 4})  $\PP^1\times S_{11-\rho}$ is not weakly $2$-Fano by  \ref{del Pezzo} and Lemma \ref{B3}.

\

%%%%%%%%%%%%%%%%%%%%%%%%%%%%%%%%%%%%%%%%%%%%%%%%%%%%%%%%%
%                                                       %
%                 SECTION 9                            %
%                                                       %
%%%%%%%%%%%%%%%%%%%%%%%%%%%%%%%%%%%%%%%%%%%%%%%%%%%%%%%%%

\section{Fano fourfolds with index $i\geq 2$ and Picard number $\rho\geq2$}\label{Section:4folds}

The classification of Fano fourfolds of index $2$ and $\rho>1$
can be found in \cite[Table 12.7]{IP}. 
We go through this list, check that none of them is $2$-Fano,
and point out the cases in which the $4$-fold is weakly $2$-Fano. 
We use the same notation as in the previous section.

\

({\bf 1}) $\PP^1\times V_1$. This is not weakly $2$-Fano by Lemma \ref{B3}. 

\

({\bf 2}) $\PP^1\times V_2$. This is not weakly $2$-Fano by Lemma \ref{B3}. 

\

({\bf 3})  $\PP^1\times V_3$. This is not weakly $2$-Fano by Lemma \ref{B3}. 

\

({\bf 4}) A double cover of $\PP^2\times\PP^2$ whose branch locus is a divisor of bidegree $(2,2)$. 
Since $\PP^1\times\PP^2$ is not $2$-Fano by Lemma \ref{B3}, this is not weakly $2$-Fano by Corollary \ref{C}(ii). 

\

({\bf 5}) A divisor of $\PP^2\times\PP^3$  of bidegree $(1,2)$. This is not weakly $2$-Fano by
Example~\ref{(a,b)}.

\

({\bf 6}) $\PP^1\times V_4$. This is not weakly $2$-Fano by Lemma \ref{B3}. 

\

({\bf 7}) An intersection of two divisors of bidegree $(1,1)$ on  $\PP^3\times\PP^3$.
This is not weakly $2$-Fano by
Example~\ref{ci_in_product}.

\

({\bf 8}) A divisor of $\PP^2\times Q^3$  of bidegree $(1,1)$.
By making a computation similar to those in \ref{H_in_product}, one can check that 
this is not weakly $2$-Fano.

\

({\bf 9}) $\PP^1\times V_5$. This is not weakly $2$-Fano by Lemma \ref{B3}.

\

({\bf 10}) The blow-up of $Q^4$ along a conic $C$ which is not contained in a plane lying on $Q^4$.
We claim that this is 
not weakly $2$-Fano.

The normal bundle of $C$ in $Q^4$ is $N\cong \cO(2)\oplus\cO(2)\oplus\cO(2)$.
Let $\pi:X\to Q^4$ denote the blow-up, and $E\cong \PP(N^*)$ the exceptional divisor.
Consider the surface $S$ in $E$,
ruled over $C$,
corresponding to a surjection
$$
N^*\cong \cO(-2)\oplus \cO(-2)\oplus \cO(-2) \onto \cO(-2)\oplus \cO(-2).
$$
Using the formula for $\ch_2$ from Lemma~\ref{blow-up}, one gets that $\ch_2(X)\cdot S=-2$.
 
\

({\bf 11}) $\PP_{\PP^3}(\cE)$, where $\cE$ is the null-correlation bundle on $\PP^3$. 
Recall that $c_1(\cE)=0$ and $c_2(\cE)=h^2$.
Therefore this is weakly $2$-Fano but not $2$-Fano by Lemma~\ref{B1}.

\

({\bf 12}) The blow-up $Q^4\subset\PP^5$ along a line $\ell$. We claim that this is 
not weakly $2$-Fano.

The normal bundle of $\ell$ in $Q^4$ is $N\cong \cO(1)\oplus\cO(1)\oplus\cO$.
Let $\pi:X\to Q^4$ denote the blow-up, and $E\cong \PP(N^*)$ the exceptional divisor.
Consider the surface $S$ in $E$,
ruled over $\ell$,
corresponding to the surjection
$$
N^*\cong \cO\oplus \cO(-1)\oplus \cO(-1) \onto \cO(-1)\oplus \cO(-1).
$$
Using the formula  for $\ch_2$ from Lemma~\ref{blow-up}, one gets that $\ch_2(X)\cdot S=-2$.

\

({\bf 13}) $\PP_{Q^3}(\cO(-1)\oplus\cO)$, where $Q^3\subset\PP^4$ is a smooth quadric.
This is weakly $2$-Fano but not $2$-Fano by Lemma~\ref{B1}.

\

({\bf 14}) $\PP^1\times\PP^3$.
This is weakly $2$-Fano but not $2$-Fano by Lemma~ \ref{B3}.

\

({\bf 15}) $\PP_{\PP^3}(\cO(-1)\oplus\cO(1))$.
This weakly $2$-Fano but not $2$-Fano by Lemma~\ref{B1}.

\

({\bf 16}) $\PP^1\times W$.
This is weakly $2$-Fano but not $2$-Fano by  Lemma~\ref{B3}.

\

({\bf 17})  $\PP^1\times V_7$.
This is weakly $2$-Fano but not $2$-Fano by  Lemma~\ref{B3}.

\

({\bf 18}) $\PP^1\times\PP^1\times\PP^1\times\PP^1$.
This is weakly $2$-Fano but not $2$-Fano by  \ref{B3}.


\begin{thebibliography}{xxxxxxxxx}


\bibliographystyle{alpha}

\bibitem[And11]{Anderson} {D. Anderson},
{{\em Chern class formulas for ${G}_2$ Schubert loci}}, 
{Trans. Amer. Math. Soc.} {Vol. \bf363} {(2011)}, {no. 12}, {6615--6646}.

\bibitem[AC12]{AC} {C. Araujo and A.-M. Castravet},
{{\em Minimal rational curves on higher Fano manifolds}}, 
{Amer. J. of Math.} {Vol. \bf134} {(2012)}, {no. 1}, {87--107}.

\bibitem[BS02]{BS} {N. Bergeron and F. Sottile},
{{\em A Pieri-type formula for isotropic flag manifolds}}, 
{Trans. Amer. Math. Soc.} {Vol. \bf354} {(2002)}, {no. 7}, {2659--2705}.

\bibitem[Cam92]{campana} {F. Campana},
{{\em Connexit\'e rationnelle des vari\'et\'es de Fano}}, 
{Ann. Sci. \'Ecole Norm. Sup.} {(4) \bf35} {(1992)}, {no. 5}, {539--545}.

\bibitem[Cos09]{Coskun} {I. Coskun},
{{\em  Restriction varieties and geometric branching rules}}. 
{Adv. Math.} {Vol. \bf228} {(2011)}, {no. 4}, {2441--2502}.

\bibitem[deJS06]{dJ-S:2fanos_1} {A. J. de Jong and J. Starr}, {{\em A note on Fano manifolds whose second Chern character is positive}},  {preprint arXiv:math/0602644v1 [math.AG]}, {(2006)}.

\bibitem[deJS07]{dJ-S:2fanos_2} {A. J. de Jong and J. Starr}, {{\em Higher Fano manifolds and rational surfaces}},  {Duke Math. J.} {Vol. \bf139} {(2007)}, {no. 1}, {173--183}.

\bibitem[deJHS08]{dJ-S:sections_over_surfaces}{A. J. de Jong, X. He, and J. Starr}, 
{{\em Families of rationally simply connected varieties over surfaces and torsors for semisimple groups}},  {preprint arXiv:0809.5224v1 [math.AG]}, {(2008)}.

\bibitem[Fuj82a]{fujita82a}{T. Fujita}, 
{{\em Classification of projective varieties of {$\Delta $}-genus one}}, 
{Proc. Japan Acad. Ser. A Math. Sci.} {Vol. \bf58} {(1982)}, {no. 3}, {113--116}.

\bibitem[Fuj82b]{fujita82b} {T. Fujita},
{{\em On polarized varieties of small {$\Delta $}-genera}}, 
{Tohoku Math. J. (2)} {Vol. \bf34} {(1982)}, {no. 3}, {319--341}.

\bibitem[Ful98]{Fulton}{W. Fulton}, {{\em Intersection theory}}. {Second Edition.} 
{Ergebnisse der Mathemayik und ihrer Grenzgebeite. 3. Folge.},
{\bf2}{Springer-Verlag, Berlin}, {1998}.

\bibitem[GH]{Griffiths-Harris}{P. Griffiths and J. Harris}, {{\em Principles of Algebraic Geometry}}. {Wiley Classics Library.} {John Wiley  Sons, Inc.}, {New York}, {1994}.

\bibitem[GHMS02]{GHS_converse}  {T. Graber, J. Harris, B. Mazur, and J. Starr},
{{\em Rational connectivity and sections of families over curves}}, 
{preprint arXiv:0809.5224v1 [math.AG]}, {(2008)}.

\bibitem[GHS03]{GHS}  {T. Graber, J. Harris, and J. Starr},
{{\em Families of rationally connected varieties}}, 
{J. Amer. Math. Soc.} {Vol. \bf16} {(2003)}, {no. 1}, {57--67}.

\bibitem[Hwa01]{hwang_ICTP}  {J.-M. Hwang},
{{\em Geometry of minimal rational curves on {F}ano manifolds}}, 
{School on Vanishing Theorems and Effective Results in Algebraic Geometry (Trieste, 2000)}, 
{ICTP Lect. Notes}{Vol. \bf6} {(2001)},  {335--393}.

\bibitem[Isk77]{Isk1} {V.A. Iskovskikh},
{{\em Fano $3$-folds I}}, {English translation},
{Math. USSR Izvestija} {Vol. \bf11} {(1977)}, {no. 3}, {485--527}.

\bibitem[Isk78]{Isk2} {V.A. Iskovskikh},
{{\em Fano $3$-folds II}}, {English translation},
{Math. USSR Izvestija} {Vol. \bf12} {(1987)}, {no. 3}, {469--506}.

\bibitem[IP99]{IP} {V.A. Iskovskikh and Yu. G. Prokhorov},
{{\em Algebraic Geometry V, Fano varieties}}, 
{Encyclopaedia Math. Sci.} {Vol. \bf47}, Springer, Berlin {(1999)}.

\bibitem[KMM92a]{kmm3}  {J. Koll\'ar, Y. Miyaoka, and S. Mori},
{{\em Rational connectedness and boundedness of Fano manifolds}}, 
{J. Differential Geom.} {Vol. \bf36} {(1992)}, {no. 3}, {765--779}.

\bibitem[KMM92b]{kmm2}  {J. Koll\'ar, Y. Miyaoka, and S. Mori},
{{\em Rationally connected varieties}}, 
{J. Algebraic Geom.} {Vol. \bf1} {(1992)}, {no. 3}, {429--448}.

\bibitem[KO73]{kobayashi_ochiai} {S. Kobayashi, and T. Ochiai},
{{\em Characterizations of complex projective spaces and hyperquadrics}}, 
{J. Math. Kyoto Univ.} {Vol. \bf13} {(1973)}, {31--47}.

\bibitem[MM81]{Mori-Mukai} {S. Mori and S. Mukai},
{{\em Classification of Fano $3$-folds with $B_2\geq2$}}, 
{Manuscripta Math.} {Vol. \bf36} {(1981)}, {147--162}.

\bibitem[MM03]{Mori-Mukai erratum} {S. Mori and S. Mukai},
{{\em Erratum: Classification of Fano $3$-folds with $B_2\geq2$}}, 
{Manuscripta Math.} {Vol. \bf110} {(2003)}, {407}.

\bibitem[Muk89]{Mukai89} {S. Mukai},
{{\em  Birational classification of Fano 3-folds and Fano manifolds of coindex $3$}}, 
{Proc. Nat. Acad. Sci. U.S.A.} {Vol. \bf86} {(1992)}, {no. 9}, {3000--3002}.

%\bibitem[Muk92]{Mukai92} {S. Mukai},
%{{\em  Fano threefolds}}, {Complex projective geometry (Trieste, 1989/Bergen 1989)}
%{255--263}, {London Math. Soc. Lect. Notes Ser.} {Vol. \bf179}, {Cambridge Uni. Press, Cambridge}, {(1992)}.

\bibitem[Muk95]{Mukai95} {S. Mukai},
{{\em Curves and symmetric spaces, I}}, 
{Amer. J. of Math.} {Vol. \bf117} {(1995)}, {no. 6}, {1627--1644}.

%\bibitem[Muk02]{Mukai02} {S. Mukai},
%{{\em New developments in the theory of Fano threefolds: 
%vector bundle method and moduli problems}}, 
%{Sugaku Expositions} {Vol. \bf15}, {no. 2}, {(2002)}, {125--150}.

\bibitem[Nob11]{nobili_toric_2-fano_4folds} {E. E. Nobili},
{{\em Classification of Toric 2-Fano 4-folds}}, 
{Bulletin of the Brazilian Mathematical Society} {Vol. \bf42}, {no. 3}, {(2011)}, {399--414}.

\bibitem[Nob12]{nobili_thesis} {E. E. Nobili}, {{\em Birational geometry of toric varieties}},  
{IMPA Ph.D. Thesis}, {arXiv:1204.3883 [math.AG]}{(2012)}.

\bibitem[Sat11]{sato_toric_2Fanos}  {H. Sato},
{{\em The numerical class of a surface on a toric manifold}}, 
{preprint arXiv:1106.5949v1 [math.AG]}, {(2011)}.

%\bibitem[SW90a]{Szurek-Wisniewski1} {M. Szurek and J. Wisniewski}, 
%{{\em Fano bundles over $\PP^3$ and $Q^3$}}, 
%{Pac. J. Math.} {Vol. \bf141} {(1990)}, {197--208}.

%\bibitem[SW90b]{Szurek-Wisniewski2} {M. Szurek and J. Wisniewski} ,
%{{\em Fano bundles of rank $2$ on $\PP^3$ and $Q^3$}}, 
%{Pac. J. Math.} {Vol. \bf141} {(1990)}, {197--208}.

%\bibitem[Wi\'s89]{Wisniewski} {J. Wisniewksi},
%{{\em Ruled Fano $4$-folds of index $2$.}}, 
%{Proc. Amer. Math. Soc.} {Vol. \bf105} {(1989)}, {55--61}.

\bibitem[Wi\'s91]{wis} {J. Wi\'sniewski}, {{\em On Fano Manifolds of Large Index}}, 
{Manuscripta Math.} {Vol. \bf70} {(1991)}, {145--152}.


\end{thebibliography}
\end{document}